\documentclass[11pt]{amsart}
\usepackage{amsmath, amsfonts, amssymb, amsbsy, graphicx,  upref}

\newtheorem{thm}{Theorem}[section]
\newtheorem{lmm}[thm]{Lemma}

\newtheorem{prop}[thm]{Proposition}

\newcommand{\cov}{\mathrm{Cov}}

\newcommand{\ee}{\mathbb{E}}

\newcommand{\mf}{\mathcal{F}}

\newcommand{\pp}{\mathbb{P}}

\newcommand{\ra}{\rightarrow}
\newcommand{\rr}{\mathbb{R}}

\newcommand{\var}{\mathrm{Var}}
\newcommand{\ve}{\varepsilon}

\newcommand{\zz}{\mathbb{Z}}

\usepackage{ifpdf}
\ifpdf%
  \usepackage{pdftricks}
  \begin{psinputs}
    \usepackage{pstricks}
  \end{psinputs}
\else
  \usepackage{pstricks}
  \newenvironment{pdfpic}{}{}
\fi

\newcommand{\ep}{\epsilon}

\usepackage{hyperref}

\begin{document}
\title[Scaling exponents in first-passage percolation]{The universal relation between scaling exponents in first-passage percolation}
\author{Sourav Chatterjee}
\address{Courant Institute of Mathematical Sciences, New York University, 251 Mercer Street, New York, NY 10012}
\thanks{Research partially supported by the NSF grant DMS-1005312}
\begin{abstract}
It has been conjectured in numerous physics papers that in ordinary first-passage percolation on integer lattices, the fluctuation exponent $\chi$ and the wandering exponent $\xi$ are related through the universal relation $\chi=2\xi -1$, irrespective of the dimension. This is sometimes called the KPZ relation between the two exponents. This article gives a rigorous proof of this conjecture assuming that the exponents exist in a certain sense. 
\end{abstract}
\maketitle

\section{Introduction}

Consider the space $\rr^d$ with Euclidean norm $|\cdot |$, where $d\ge 2$. Consider $\zz^d$ as a subset of this space, and say that two points $x$ and $y$ in $\zz^d$ are nearest neighbors if $|x-y|=1$. Let $E(\zz^d)$ be the set of nearest neighbor bonds in $\zz^d$. Let $t= (t_e)_{e\in E(\zz^d)}$ be a collection of i.i.d.\ non-negative random variables. In first-passage percolation, the variable $t_e$ is usually called the `passage time' through the edge $e$, alternately called the `edge-weight' of $e$. We will sometimes refer to the collection $t$ of edge-weights as the `environment'. The total passage time, or total weight, of a path $P$ in the environment $t$ is simply the sum of the weights of the edges in $P$ and will be denoted by $t(P)$ in this article. The first-passage time $T(x,y)$ from a point $x$ to a point $y$ is the minimum total passage time among all lattice paths from $x$ to $y$. For all our purposes, it will suffice to consider self-avoiding paths; henceforth, `lattice path' will refer to only self-avoiding paths.

Note that if the edge-weights are continuous random variables, then with probability one there is a unique `geodesic' between any two points $x$ and $y$. This is  denoted by $G(x,y)$ in this paper. Let $D(x,y)$ be the maximum deviation (in Euclidean distance) of this path from the straight line segment joining $x$ and $y$ (see Figure \ref{fig0}).

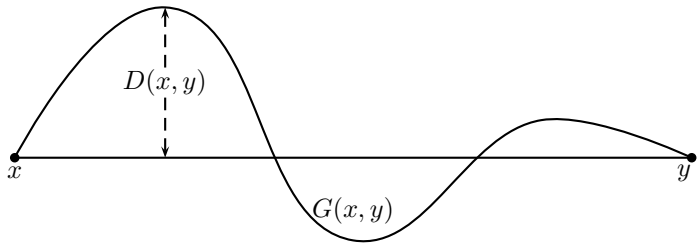
\begin{figure}[h]
\begin{pdfpic}
\begin{pspicture}(-1,-1)(10,2.5)
\psset{xunit=1cm,yunit=1cm}
\rput(0,-.2){\small $x$}
\rput(8.9,-.2){\small $y$}
\psline{*-*}(0,0)(9,0)
\pscurve{-}(0,0)(2,2)(4.5,-1.1)(7,.5)(9,0)
\psline[linestyle=dashed]{->}(2,1.2)(2,2)
\rput(2,1){\small $D(x,y)$}
\psline[linestyle=dashed]{->}(2, .8)(2,0)
\rput(4.5,-.7){\small $G(x,y)$}
\end{pspicture}
\end{pdfpic}
\caption{The geodesic $G(x,y)$ and the deviation $D(x,y)$.}
\label{fig0}
\end{figure}

Although invented by mathematicians \cite{hw65}, the first-passage percolation and related models have  attracted considerable attention in the theoretical physics literature (see \cite{ks91} for a survey). Among other things, the physicists are particularly interested in two `scaling exponents', sometimes denoted by $\chi$ and $\xi$ in the mathematical physics literature. The {\it fluctuation exponent} $\chi$ is a number that quantifies the order of fluctuations of the first-passage time $T(x,y)$. Roughly speaking, for any $x, y$,  
\[
\text{the typical value of } T(x,y)-\ee T(x,y) \text{ is of the order } |x-y|^\chi. 
\]
The {\it wandering exponent} $\xi$ quantifies the magnitude of $D(x,y)$. Again, roughly speaking, for any $x,y$,
\[
\text{the typical value of } D(x,y) \text{ is of the order } |x-y|^\xi. 
\]
There have been several attempts to give precise  mathematical definitions for these exponents (see \cite{lnp96} for some examples) but I could not find a consensus in the literature. The main hurdle is that no one knows whether the exponents actually exist, and if they do, in what sense.

There are many conjectures related to $\chi$ and $\xi$. The main among these, to be found in numerous physics papers \cite{hh85,  kpz86, kz87, krug87, km89, ks91,   meakin86, medina89, wk87}, including the famous paper of Kardar, Parisi and Zhang \cite{kpz86}, is that although $\chi$ and $\xi$ may depend on the dimension, they always satisfy the relation 
\begin{equation*}\label{kpz}
\chi = 2\xi - 1.
\end{equation*}
A well-known conjecture from \cite{kpz86} is that when $d=2$, $\chi = 1/3$ and $\xi = 2/3$. Yet another belief is that $\chi = 0$ if $d$ is sufficiently large. Incidentally, due to its connection with \cite{kpz86}, I've heard in private conversations the relation $\chi=2\xi -1$ being referred to as the `KPZ relation' between $\chi$ and~$\xi$. 

There are a number of rigorous results for $\chi$ and $\xi$, mainly from the late eighties and early nineties. 
One of the first non-trivial results is due to Kesten \cite[Theorem 1]{kesten93}, who proved that $\chi \le 1/2$ in any dimension. The only improvement on Kesten's result till date  is due to Benjamini, Kalai and Schramm \cite{bks03}, who  proved that for first-passage percolation in $d\ge 2$ with binary edge-weights, 
\begin{equation}\label{bksineq}
\sup_{v\in \zz^d, \ |v| >1} \frac{\var T(0,v)}{|v|/\log |v|} < \infty. 
\end{equation}
Bena\"im and Rossignol \cite{br08} extended this result to a large class of edge-weight distributions that they call `nearly gamma' distributions. The definition of a nearly gamma distribution is as follows. A positive random variable $X$ is said to have a nearly gamma distribution if it has a continuous probability density function $h$  supported on an interval $I$ (which may be unbounded), and its distribution function $H$ satisfies, for all $y\in I$, \[
\Phi'\circ \Phi^{-1} (H(y)) \le A \sqrt{y} h(y),
\] 
for some constant $A$, where $\Phi$ is the distribution function of the standard normal distribution. Although the definition may seem a bit strange, Bena\"im and Rossignol \cite{br08} proved  that this class is actually quite large, including e.g.\ exponential, gamma, beta and uniform distributions on intervals. 

The only non-trivial lower bound on the fluctuations of passage times is due to Newman and Piza \cite{np95} and Pemantle and Peres \cite{pp94}, who showed that in $d=2$, $\var T(0,v)$ must grow at least as fast as $\log|v|$. Better lower bounds can be proved if one can show that with high probability, the geodesics lie in `thin cylinders' \cite{cd10}. 

For the wandering exponent $\xi$, the main rigorous results are due to Licea, Newman and Piza \cite{lnp96} who showed that $\xi^{(2)} \ge 1/2$ in any dimension, and $\xi^{(3)} \ge 3/5$ when $d=2$, where $\xi^{(2)}$ and $\xi^{(3)}$ are exponents defined in their paper which may be equal to $\xi$.

Besides the bounds on $\chi$ and $\xi$ mentioned above, there are some rigorous results relating $\chi$ and $\xi$ through inequalities. Wehr and Aizenman~\cite{wa90}  proved the inequality  $\chi \ge (1-(d-1)\xi)/2$ in a related model, and the version of this inequality for first-passage percolation was proved by Licea, Newman and Piza~\cite{lnp96}. The closest that anyone came to proving $\chi=2\xi -1$ is a result of Newman and Piza \cite{np95}, who proved that $\chi' \ge 2\xi -1$, where $\chi'$ is a related exponent which may be equal to $\chi$. This has also been observed by Howard \cite{howard04} under different assumptions. 

Incidentally, in the model of Brownian motion in a Poissonian potential, W\"uthrich \cite{wuthrich98} proved the equivalent of the KPZ relation assuming that the exponents exist. 

The following theorem establishes the relation $\chi = 2\xi -1$ assuming that the exponents $\chi$ and $\xi$ exist in a certain sense (to be defined in the statement of the theorem) and that the distribution of edge-weights is nearly gamma.
\begin{thm}\label{kpzthm}
Consider  the first-passage percolation model on $\zz^d$, $d\ge 2$, with i.i.d.\ edge-weights. Assume  that the distribution of edge-weights is `nearly gamma' in the sense of Bena\"im and Rossignol \cite{br08} (which includes exponential, gamma, beta and uniform distributions, among others), and has a finite moment generating function in a neighborhood of zero. Let $\chi_a$ and $\xi_a$ be the smallest real  numbers such that for all $\chi'> \chi_a$ and $\xi'> \xi_a$, there exists $\alpha > 0$ such that 
\begin{align}
&\sup_{v\in \zz^d\backslash\{0\}} \ee\exp\biggl(\alpha \frac{|T(0,v) - \ee T(0,v)|}{|v|^{\chi'}}\biggr) < \infty, \tag{A1} \label{up1}\\
&\sup_{v\in \zz^d\backslash\{0\}} \ee\exp\biggl(\alpha \frac{D(0,v)}{|v|^{\xi'}}\biggr) < \infty. \tag{A2} \label{up2}
\end{align}
Let $\chi_b$ and $\xi_b$ be the largest real numbers such that for all $\chi'< \chi_b$ and  $\xi'< \xi_b$, there exists $C > 0$ such that 
\begin{align}
&\inf_{v\in \zz^d, \ |v| > C} \frac{\var(T(0,v))}{|v|^{2{\chi'}}} > 0, \tag{A3} \label{down1}\\
&\inf_{v\in \zz^d, \ |v| > C} \frac{\ee D(0,v)}{|v|^{\xi'}} > 0. \tag{A4} \label{down2}
\end{align}
Then $0\le \chi_b \le \chi_a \le 1/2$, $0\le \xi_b\le \xi_a \le 1$ and $\chi_a \ge 2\xi_b -1$. Moreover, if it so happens that $\chi_a=\chi_b$ and $\xi_a = \xi_b$, and these two numbers are denoted by $\chi$ and $\xi$, then they must necessarily satisfy the relation $\chi = 2\xi - 1$. 
\end{thm}

Note that if $\chi_a=\chi_b$ and $\xi_a = \xi_b$ and these two numbers are denoted by $\chi$ and $\xi$, then $\chi$ and $\xi$ are characterized by the properties that for every $\chi' >\chi$ and $\xi' > \xi$, there are some positive $\alpha$ and $C$ such that for all $v\ne 0$, 
\begin{align*}
&\ee\exp\biggl(\alpha \frac{|T(0,v) - \ee T(0,v)|}{|v|^{\chi'}}\biggr) < C \ \ \text{and} \ \  \ee\exp\biggl(\alpha \frac{D(0,v)}{|v|^{\xi'}}\biggr) < C, 
\end{align*}
and for every $\chi' < \chi$ and $\xi'< \xi$ there are some positive $B$ and $C$ such that for all $v$ with $|v| > C$, 
\begin{align*}
&\var(T(0,v)) > B|v|^{2{\chi'}} \ \ \text{and} \ \ \ee D(0,v) > B|v|^{\xi'}. 
\end{align*}
It seems reasonable to expect that if the two exponents $\chi$ and $\xi$ indeed exist, then they should satisfy the above properties.

Incidentally, a few months after the first draft of this paper was put up on arXiv, Auffinger and Damron \cite{ad11} were able to replace a crucial part of the proof of Theorem \ref{kpzthm} with a simpler argument that allowed them to remove the assumption that the edge-weights are nearly-gamma.

Section \ref{sketch} has a sketch of the proof of Theorem \ref{kpzthm}. The rest of the paper is devoted to the actual proof. Proving that $0\le \chi_b\le \chi_a \le1/2$ and  $0\le \xi_b \le \xi_a\le1$ is a routine exercise; this is done in Section \ref{trivial}. Proving that $\chi_a \ge 2\xi_b -1$ is also relatively easy and similar to the existing proofs of analogous inequalities, e.g.\ in \cite{np95, howard04}. This is done in Section~\ref{easypart}. The `hard part' is proving the opposite inequality, that is, $\chi\le 2\xi -1$ when $\chi=\chi_a=\chi_b$ and $\xi=\xi_a=\xi_b$. This is done in Sections \ref{012}, \ref{hard2} and \ref{hard3}. 


\section{Proof sketch}\label{sketch}
I will try to give a sketch of the proof in this section. I have found it very hard to aptly summarize the main ideas in the proof without going into the details. This proof-sketch represents the end-result of my best efforts in this direction. If the interested reader finds the proof sketch too obscure, I would like to request him to return to this section after going through the complete proof, whereupon this high-level sketch may shed some illuminating insights. 

Throughout this proof sketch, $C$ will denote any positive constant that depends only on the edge-weight distribution and the dimension. 
Let $h(x) := \ee(T(0,x))$.  The function $h$ is subadditive. Therefore the limit
\[
g(x) := \lim_{n\ra\infty} \frac{h(nx)}{n}
\]
exists for all $x\in \zz^d$. The definition can be extended to all $x\in \mathbb{Q}^d$ by taking $n\ra\infty$ through a subsequence, and can be further extended to all $x\in \rr^d$ by uniform continuity. The function $g$ is a norm on $\rr^d$.

The function $g$ is a norm, and hence much more well-behaved than $h$. If $|x|$ is large, $g(x)$ is supposed to be a good approximation of $h(x)$. A method developed by Ken Alexander \cite{alexander93, alexander97} uses the order of fluctuations of passage times to infer bounds on $|h(x)-g(x)|$. In the setting of Theorem~\ref{kpzthm}, Alexander's method yields that for any $\ve >0$, there exists $C$ such that for all $x\ne 0$,
\begin{equation}\label{alex}
g(x)\le h(x)\le g(x)+C|x|^{\chi_a + \ve}.
\end{equation}
This is formally recorded in Theorem \ref{gapthm}. 
In the proof of the main result, the above approximation will allow us to replace the expected passage time $h(x)$ by the norm $g(x)$.

In Lemma \ref{curve}, we prove that there is a unit vector $x_0$ and a hyperplane $H_0$ perpendicular to $x_0$ such that for some $C>0$, for all $z\in H_0$, 
\[
|g(x_0+z)-g(x_0)| \le C|z|^2.
\]
Similarly, there is a unit vector $x_1$ and a hyperplane $H_1$ perpendicular to $x_1$ such that for some $C>0$, for all $z\in H_1$, $|z|\le 1$, 
\[
g(x_1+z)\ge g(x_1) + C|z|^2.
\]
The interpretations of these two inequalities is as follows. In the direction $x_0$, the unit sphere of the norm $g$  is `at most as curved as an Euclidean sphere' and in the direction $x_1$, it is `at least as curved as an Euclidean sphere'. 

Now take a look at Figure \ref{figx1}. Think of $m$ as a fraction of $n$. By the definition of the direction of curvature~$x_1$ and Alexander's approximation~\eqref{alex}, for any $\ve >0$, 
\begin{align*}
&\text{Expected passage time of the path $P$}\\
&\ge g(mx_1+z) + g(nx_1 - (mx_1+z)) + O(n^{\chi+\ve})\\
&= m g(x_1 + z/m) + (n-m) g(x_1 + z/(n-m)) + O(n^{\chi + \ve})\\
&\ge n g(x_1) + C|z|^2/n + O(n^{\chi+\ve})\\
&\ge \ee (T(0, nx_1))+ C|z|^2/n + O(n^{\chi+\ve}).
\end{align*}
Suppose $|z|=n^\xi$. Then $|z|^2/n = n^{2\xi-1}$. Fluctuations of $T(0,nx_1)$ are of order $n^\chi$. Thus, if $2\xi-1 > \chi$, then $P$ cannot be a geodesic from $0$ to $nx_1$. This sketch is formalized into a rigorous argument in Section \ref{easypart} to prove that $\chi_a \ge 2\xi_b -1$. 
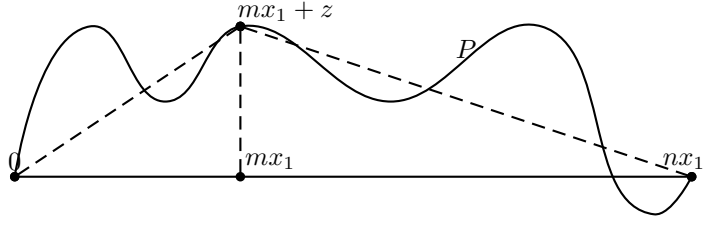
\begin{figure}[h]
\begin{pdfpic}
\begin{pspicture}(-1,-1)(10,2.5)
\psset{xunit=1cm,yunit=1cm}
\rput(0,.2){\small $0$}
\rput(8.9,.2){\small $nx_1$}
\psline{*-*}(0,0)(9,0)
\psline[linestyle=dashed]{*-*}(3,0)(3,2)
\rput(3.4, 0.2){\small $mx_1$}
\rput(3.6, 2.2){\small $mx_1+z$}
\psline[linestyle=dashed]{*-*}(0,0)(3,2)
\psline[linestyle=dashed]{*-*}(3,2)(9,0)
\pscurve(0,0)(1,2)(2,1)(3,2)(5,1)(7,2)(8.5,-.5)(9,0)
\rput(6, 1.7){\small $P$}
\end{pspicture}
\end{pdfpic}
\caption{Proving $\chi \ge 2\xi -1$}
\label{figx1}
\end{figure}

Next, let me sketch the proof of $\chi \le 2\xi -1$ when $\chi > 0$. The methods developed in \cite{cd10} for first-passage percolation in thin cylinders have some bearing on  this part of the proof. Recall the direction of curvature $x_0$. Let $a= n^\beta$, $\beta<1$.  Let $m = n/a = n^{1-\beta}$. Under the conditions $\chi > 2\xi -1$ and $\chi > 0$, we will  show that there is a $\beta < 1$ such that
\[
T(0,nx_0) = \sum_{i=0}^{m-1} T(ia x_0, (i+1)ax_0) + o(n^{\chi}). \tag{$\star$}
\]
This will lead to a contradiction, as follows. Let $f(n) := \var T(0, nx_0)$. Then by Bena\"im and Rossignol \cite{br08}, $f(n)\le Cn/\log n$.
Under ($\star$), by the Harris-FKG inequality, 
\begin{align*}
f(n) = \var T(0,nx_0) &\ge m \var T(0, ax_0) + o(n^{2\chi})\\
&= n^{1-\beta} f(n^\beta) + o(n^{2\chi}).  
\end{align*}
If $\beta$ is chosen sufficiently small, the first term on the right will dominate the second. Consequently,  
\[
\liminf_{n\ra\infty} \frac{f(n)}{n^{1-\beta}f(n^\beta)} \ge 1. \tag{$\dagger$}
\]
Choose $n_0>1$  and define $n_{i+1} = n_i^{1/\beta}$ for each $i$. Let $v(n) := f(n)/n$. Then $v(n_i)\le C/\log n_i \le C \beta^i$. But by ($\dagger$), $\liminf v(n_{i+1})/v(n_i)\ge 1$, and so for all $i$ large enough, $v(n_{i+1})\ge \beta^{1/2}v(n_i)$.  In particular, there is a positive constant $c$ such that for all $i$, $v(n_i) \ge c \beta^{i/2}$.  Since $\beta < 1$, this gives a contradiction for $i$ large, therefore proving that $\chi\le 2\xi-1$.

Let me now sketch a proof of ($\star$) under the conditions $\chi > 2\xi -1$ and $\chi > 0$. Let $a=n^\beta$ and $b= n^{\beta'}$, where $\beta' < \beta<1$. Consider a cylinder of width $n^\xi$ around the line joining $0$ and $nx_0$. Partition the cylinder into alternating big and small cylinders of widths $a$ and $b$ respectively.  Call the boundary walls of these cylinders $U_0, V_0, U_1, V_1, \ldots, V_{m-1}, U_m$, where $m$ is roughly $n^{1-\beta}$ (see Figure \ref{figx3}). 

\begin{figure}[h]
\begin{pdfpic}
\begin{pspicture}(-1,-1)(10,2.5)
\psset{xunit=1cm,yunit=1cm}
\pspolygon(0,0)(2,0)(2,1)(0,1)
\pspolygon[fillstyle=solid, fillcolor=lightgray](2,0)(3,0)(3,1)(2,1)
\pspolygon(3,0)(5,0)(5,1)(3,1)
\pspolygon[fillstyle=solid, fillcolor=lightgray](5,0)(6,0)(6,1)(5,1)
\psline{-}(6,0)(6.5,0)
\psline{-}(6,1)(6.5,1)
\psline[linestyle=dotted]{-}(6.5,0)(7.5,0)
\psline[linestyle=dotted]{-}(6.5,1)(7.5,1)
\psline{-}(7.5,0)(8,0)
\psline{-}(7.5,1)(8,1)
\pspolygon[fillstyle=solid, fillcolor=lightgray](8,0)(9,0)(9,1)(8,1)
\rput(0,-.2){\small $U_0$}
\rput(2,-.2){\small $V_0$}
\rput(3,-.2){\small $U_1$}
\rput(5,-.2){\small $V_1$}
\rput(6,-.2){\small $U_2$}
\rput(8,-.2){\small $V_{m-1}$}
\rput(9,-.2){\small $U_m$}
\psline[linestyle=dotted]{*-*}(0,.5)(0,.5)
\psline[linestyle=dotted]{*-*}(9,.5)(9,.5)
\rput(-.2,.5){\small $0$}
\rput(9.4, .5){\small $nx_0$}
\psline{<-}(0,0.5)(.75, 0.5)
\rput(.9,0.5){\small $a$}
\psline{->}(1.05, .5)(2,.5)
\psline{<-}(2,0.5)(2.35, 0.5)
\rput(2.5,0.5){\small $b$}
\psline{->}(2.65, .5)(3,.5)
\end{pspicture}
\end{pdfpic}
\caption{Cylinder of width $n^\xi$ around the line joining $0$ and $nx_0$}
\label{figx3}
\end{figure}
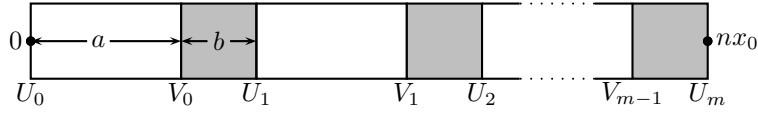

Let $G_i := G(U_i, V_i)$, that is, the path with minimum passage time between any vertex in $U_i$ and any vertex in $V_i$.  Let $u_i$ and $v_i$ be the endpoints of $G_i$.  Let $G_i' := G(v_i, u_{i+1})$.  The concatenation of the paths  $G_0'$, $G_1$, $G_1'$, $G_2$, $\ldots$, $G_{m-1}'$, $G_m$ is a path from $U_0$ to $U_{m}$. Therefore,
\begin{align*}
T(U_0, U_{m}) &\le\sum_{i=1}^{m-1} T(U_i, V_i) + \sum_{i=0}^{m-1} T(v_i, u_{i+1}). 
\end{align*}
Next, let $G := G(U_0, U_m)$.  Let $u_i'$ be the first vertex in $U_i$ visited by $G$ and let $v_i'$ be the first vertex in $V_i$ visited by $G$. If $G$ stays within the cylinder throughout, then $T(u_i', v_i') \ge T(U_i, V_i)$ and $T(v_i', u_{i+1}')\ge T(V_i, U_{i+1})$. Thus,
\begin{align*}
T(U_0, U_{m}) &\ge \sum_{i=0}^{m-1} T(U_i, V_i) + \sum_{i=0}^{m-1} T(V_i, U_{i+1}). 
\end{align*}
Thus, if $G(U_0, U_m)$ stays in a cylinder of width $n^\xi$, then  
\begin{align*}
0&\le T(U_0,U_m)-\sum_{i=0}^{m-1} (T(U_i, V_i)+T(V_i, U_{i+1})) \\
&\le \sum_{i=0}^{m-1}(T(v_i, u_{i+1})-T(V_i, U_{i+1})).
\end{align*}
Therefore, 
\begin{align*}
&\biggl|T(U_0, U_m) -\sum_{i=0}^{m-1} (T(U_i, V_i) + T(V_i, U_{i+1})) \biggr|\le \sum_{i=0}^{m-1} M_i,
\end{align*}
where $M_i := \max_{v,v'\in V_i, \ u,u'\in U_{i+1}} |T(v,u) - T(v', u')|$.
Note that the errors $M_i$ come only from the small blocks. By curvature estimate in direction $x_0$, for any $v,v'\in V_i$ and $u,u'\in U_{i+1}$, 
\[
|\ee T(v,u) - \ee T(v',u')| \le C(n^\xi)^2/n^{\beta'} = Cn^{2\xi-\beta'}. 
\]
Fluctuations of $T(v,u)$ are of order $n^{\beta'\chi}$. If $2\xi-1 < \chi$, then we can choose $\beta'$ so close to $1$ that $2\xi-\beta' < \beta' \chi$. That is,  fluctuations dominate while estimating $M_i$. Consequently, $M_i$ is of order $n^{\beta' \chi}$. Thus, total error $= n^{1-\beta+\beta'\chi}$. Since $\beta' < \beta$ and $\chi >0$, this gives us the opportunity of choosing $\beta', \beta$ such that the exponent is $< \chi$. This proves ($\star$) for passage times from `boundary to boundary'. Proving ($\star$) for `point to point' passage times is only slightly more complicated.  The program is carried out in Sections \ref{012} and \ref{hard2}. 

Finally, for the case $\chi = 0$, we have to prove that $\xi \ge 1/2$.  This was proved by Licea, Newman and Piza \cite{lnp96} for a different definition of the wandering exponent. The argument does not seem to work with our definition. A proof is given in Section \ref{hard3}; I will omit this part from the proof sketch.

\section{A priori bounds}\label{trivial}
In this section we prove the a priori bounds $0\le \chi_b\le \chi_a \le1/2$ and $0\le \xi_b \le \xi_a \le 1$. First, note that the inequalities $\chi_b \le \chi_a$ and $\xi_b\le \xi_a$ are easy. For example, if $\chi_b > \chi_a$, then for any $\chi_a < \chi'< \chi''< \chi_b$,  \eqref{up1} implies that 
\[
\sup_{v\in \zz^d\backslash \{0\}} \frac{\var(T(0,v))}{|v|^{2\chi'}} < \infty,
\]
and hence for any sequence $v_n$ such that $|v_n|\ra\infty$, 
\[
\lim_{n\ra\infty} \frac{\var(T(0,v_n))}{|v_n|^{2\chi''}} = 0, 
\]
which contradicts \eqref{down1}. A similar argument shows that $\xi_b \le \xi_a$. 

To show that $\chi_b \ge 0$, let $E_0$ denote the set of all edges incident to the origin. Let $\mf_0$ denote the sigma-algebra generated by $(t_e)_{e\not\in E_0}$. Since the edge-weight distribution is non-degenerate, there exists $c_1 < c_2$ such that for an edge $e$, $\pp(t_e < c_1) >0$ and $\pp(t_e > c_2) >0$. Therefore,
\begin{equation}\label{ppmax}
\pp(\max_{e\in E_0} t_e < c_1) > 0, \ \ \pp(\min_{e\in E_0} t_e > c_2) > 0.
\end{equation}
Let $(t'_e)_{e\in E_0}$ be an independent configuration of edge weights. Define $t_e' = t_e$ if $e\not\in E_0$. Let $T'(0,v)$ be the first-passage time from $0$ to a vertex $v$ in the new environment $t'$. If $t_e < c_1$ and $t_e' > c_2$ for all $e\in E_0$, then $T'(0,v) > T(0,v) + c_2-c_1$. Thus, by \eqref{ppmax}, there exists $\delta > 0$ such that for any $v$ with $|v| \ge 2$,
\[
\ee \var(T(0,v)|\mf_0)  = \frac{1}{2}\ee(T(0,v)-T'(0,v))^2 > \delta.
\]
Therefore $\var(T(0,v)) > \delta$ and so $\chi_b \ge 0$. 

To show that $\xi_b \ge 0$, note that there is an $\ep>0$ small enough such that for any $v\in \zz^d$ with $|v|\ge 2$, there can be at most one lattice path from $0$ to $v$ that stays within distance $\ep$ from the straight line segment joining $0$ to $v$. Fix such a vertex $v$ and such a path $P$. If the number of edges in $P$ is sufficiently large, one can use the non-degeneracy of the edge-weight distribution to show by an explicit assignment of edge weights that 
\[
\pp(\text{$P$ is a geodesic}) < \delta,
\] 
where $\delta< 1$ is a constant that depends only on the edge-weight distribution (and not on $v$ or $P$). This shows that for $|v|$ sufficiently large, $\ee D(0,v)$ is bounded below by a positive constant that does not depend on $v$, thereby proving that $\xi_b \ge 0$.

Let us next show that $\chi_a \le 1/2$. Essentially, this follows from \cite[Theorem 1]{kesten93} or  \cite[Proposition 8.3]{talagrand95}, with a little bit of extra work. Below, we give a proof using \cite[Theorem 5.4]{br08}. First, note that there is a constant $C_0$ such that for all $v$,
\begin{align}\label{tv1}
\ee T(0,v)\le C_0|v|_1, 
\end{align}
where $|v|_1$ is the $\ell_1$ norm of $v$. 
From the assumptions about the distribution of edge-weights, \cite[Theorem 5.4]{br08} implies that there are positive constants $C_1$ and $C_2$ such that for any $v\in \zz^d$ with $|v|_1\ge 2$, and any $0\le t\le |v|_1$, 
\begin{equation}\label{uptail}
\pp\biggl(|T(0,v)-\ee T(0,v)| \ge t\sqrt{\frac{|v|_1}{\log|v|_1}}\biggr) \le C_1 e^{-C_2 t}. 
\end{equation}
Fix a path $P$ from $0$ to $v$ with $|v|_1$ edges. Recall that $t(P)$ denotes the sum of the weights of the edges in $P$.  Since the edge-weight distribution has finite moment generating function in a neighborhood of zero and~\eqref{tv1} holds, it is easy to see that there are positive constants~$C_3$, $C_4$ and $C_4'$  such that if $|v|_1 > C_3$, then for any $t> |v|_1$, 
\begin{equation}\label{uptail2}
\begin{split}
&\pp\biggl(|T(0,v)-\ee T(0,v)| \ge t\sqrt{\frac{|v|_1}{\log|v|_1}}\biggr)\\
 &\le  \pp\biggl(T(0,v) \ge C_0|v|_1 + t\sqrt{\frac{|v|_1}{\log|v|_1}}\biggr)\\
&\le\pp\biggl(t(P) \ge C_0|v|_1 + t\sqrt{\frac{|v|_1}{\log|v|_1}} \biggr) \le e^{C_4|v|_1 - C_4't\sqrt{|v|_1/\log|v|_1}}. 
\end{split}
\end{equation}
Combining \eqref{uptail} and \eqref{uptail2} it follows that there are constants $C_5$, $C_6$ and $C_7$ such that for any~$v$ with $|v|_1> C_5$,
\[
\ee\exp\biggl(C_6 \frac{|T(0,v)-\ee T(0,v)|}{\sqrt{|v|_1/\log|v|_1}}\biggr) \le C_7. 
\]
Appropriately increasing $C_7$, one sees that the above inequality holds for all $v$ with $|v|_1\ge 2$. In particular, $\chi_a \le 1/2$. 

Finally, let us prove that $\xi_a \le 1$. Consider a self-avoiding path $P$ starting at the origin, containing $m$ edges. By the strict positivity of the edge-weight distributions, for any edge~$e$, 
\[
\lim_{\theta \ra\infty}\ee(e^{- \theta t_e}) =0.
\]
Now, for any $\theta, c>0$,
\begin{align*}
\pp(t(P) \le cm) = \pp(e^{-t(P)/c}\ge e^{-m}) &\le (e\ee(e^{-t_e/c}))^m. 
\end{align*}
Thus, given any $\delta >0$ there exists $c$ small enough such that for any $m$ and any self-avoiding path $P$ with $m$ edges,
\[
\pp(t(P) \le cm) \le \delta^m. 
\]
Since there are at most $(2d)^m$ paths with $m$ edges, therefore there exists $c$ small enough such that 
\[
\pp(t(P) \le cm \text{ for some $P$ with $m$ edges}) \le 2^{-m-1},
\]
and therefore 
\begin{equation}\label{tp}
\pp(t(P) \le cm \text{ for some $P$ with $\ge m$ edges}) \le 2^{-m}. 
\end{equation}
There is a constant $B>0$ such that for any $t\ge 1$ and any vertex $v\ne 0$, if $D(0,v) \ge t|v|$, then $G(0,v)$ has at least $Bt|v|$ edges. Therefore from \eqref{tp}, 
\begin{align*}
\pp(D(0,v) \ge t|v|) &\le \pp(T(0,v) \ge Bt|v|/c) + 2^{-Bt|v|}. 
\end{align*}
As in \eqref{uptail2}, there is a constant $C$ such that if $P$ is a path from $0$ to $v$ with $|v|_1$ edges, 
\[
\pp(T(0,v) \ge Bt|v|/c) \le \pp(t(P) \ge Bt |v|/c) \le e^{C|v| - Bt|v|/c}. 
\]
Combining the last two displays shows that for some $\alpha$ small enough,
\[
\sup_{v\ne 0}\ee\exp\biggl(\alpha \frac{D(0,v)}{|v|}\biggr) < \infty,
\]
and thus, $\xi_a \le 1$.

\section{Alexander's subadditive approximation theory}
The first step in the proof of Theorem \ref{kpzthm} is to find a suitable approximation of $\ee T(0,x)$ by a convex function $g(x)$. For $x\in \zz^d$, define 
\begin{equation}\label{hdef}
h(x) := \ee T(0,x). 
\end{equation}
It is easy to see that $h$ satisfies the subadditive inequality 
\[
h(x+y)\le h(x)+h(y). 
\]
By the standard subadditive argument, it follows that 
\begin{equation}\label{gdef}
g(x) := \lim_{n\ra\infty} \frac{h(nx)}{n}  
\end{equation}
exists for each $x\in \zz^d$. In fact, $g(x)$ may be defined similarly for $x\in \mathbb{Q}^d$ by taking $n\ra\infty$ through a sequence of $n$ such that $nx\in \zz^d$. The function $g$ extends continuously to the whole of $\rr^d$, and the extension is a norm on $\rr^d$ (see e.g.\ \cite[Lemma 1.5]{alexander97}).  Note that by subadditivity,
\begin{equation}\label{gh}
g(x)\le h(x) \text{ for all } x\in \zz^d. 
\end{equation}
Since the edge-weight distribution is continuous in the setting of Theorem~\ref{kpzthm}, it follows by a well-known result (see \cite{kesten86}) that $g(x) > 0$ for each $x\ne 0$. Let $e_i$ denote the $i$th coordinate vector in $\rr^d$. Since $g$ is symmetric with respect to interchange of coordinates and reflections across all coordinate hyperplanes, it is easy to show using subadditivity that
\begin{equation}\label{ge}
|x|_\infty \le g(x)/g(e_1) \le |x|_1 \text{ for all } x\ne 0,
\end{equation}
where $|x|_p$ denotes the $\ell_p$ norm of the vector $x$.

How well does $g(x)$ approximate $h(x)$? Following the work of Kesten \cite{kesten86, kesten93}, Alexander \cite{alexander93, alexander97} developed a general theory for tackling such questions. One of the main results of Alexander \cite{alexander97} is that under appropriate hypotheses on the edge-weights, there exists some $C >0$ such that for all $x\in \zz^d\backslash\{0\}$, 
\[
g(x)\le h(x) \le g(x) + C|x|^{1/2}\log |x|.
\]
Incidentally, Alexander has recently been able to obtain slightly improved results for nearly gamma edge-weights \cite{alexander11}. It turns out that under the hypotheses of Theorem \ref{kpzthm}, Alexander's argument goes through almost verbatim to yield the following result.
\begin{thm}\label{gapthm}
Consider the setup of Theorem \ref{kpzthm}. Let $g$ and $h$ be defined as in \eqref{gdef} and \eqref{hdef} above. Then for any $\chi' >\chi_a$, there exists $C > 0$ such that for all $x\in \zz^d$ with $|x|>1$,
\[
g(x) \le h(x)\le g(x)+C|x|^{\chi'} \log |x|. 
\]
\end{thm}
Sacrificing brevity for the sake of completeness, I will now prove Theorem~\ref{gapthm} by copying Alexander's argument with only minor changes at the appropriate points.

Fix $\chi'> \chi_a$. Since $0\le \chi_a\le 1/2$, so $\chi'$ can be chosen to satisfy $0<\chi'<1$. 

Let $B_0 := \{x: g(x)\le 1\}$. Given $x\in \rr^d$, let $H_x$ denote a hyperplane tangent to the boundary of $g(x)B_0$ at $x$. Note that if the boundary is not smooth, the choice of $H_x$ may not be unique. Let $H_x^0$ be the hyperplane through the origin that is parallel to $H_x$. There is a unique linear functional $g_x$ on $\rr^d$ satisfying
\[
g_x(y) = 0 \text{ for all } y\in H_x^0, \ \ g_x(x) = g(x). 
\]
For each $x\in \rr^d$, $C>0$ and $K > 0$ let
\begin{align*}
&Q_x(C, K) \\
&\quad := \{y\in \zz^d: |y|\le K|x|, \ g_x(y)\le g(x),\ h(y) \le g_x(y) + C |x|^{\chi'}\log|x|\}. 
\end{align*}
The following key result is taken from \cite{alexander97}. 
\begin{lmm}[Alexander \cite{alexander97}, Theorem 1.8]\label{alexander}
Consider the setting of Theorem~\ref{gapthm}. Suppose that for some $M> 1$, $C>0$, $K>0$ and $a> 1$, the following holds. For each $x\in \mathbb{Q}^d$ with $|x|\ge M$, there exists an integer $n\ge 1$, a lattice path $\gamma$ from $0$ to $nx$, and a sequence of sites $0=v_0, v_1,\ldots, v_m = nx$ in $\gamma$ such that $m\le an$ and $v_i - v_{i-1}\in Q_x(C, K)$ for all $1\le i\le m$. Then the conclusion of Theorem \ref{gapthm} holds. 
\end{lmm}
Before proving that the conditions of Lemma \ref{alexander} hold, we need some preliminary definitions and results. Define 
\[
s_x(y) := h(y) - g_x(y), \ \ y\in \zz^d. 
\]
By the definition of $g_x$ and the fact that $g$ is a norm, it is easy to see that
\begin{equation}\label{gg}
|g_x(y)|\le g(y),
\end{equation}
and by subadditivity, $g(y) \le h(y)$. Therefore $s_x(y)\ge 0$. Again from subadditivity of $h$ and linearity of $g_x$,
\begin{equation}\label{ss}
s_x(y+z)\le s_x(y) + s_x(z) \ \ \text{for all } y,z\in \zz^d. 
\end{equation}
Let $C_1 := 320 d^2/\alpha$, where $\alpha$ is from the statement of Theorem \ref{kpzthm}. As in~\cite{alexander97}, define
\begin{align*}
Q_x &:= Q_x(C_1, 2d+1),\\
G_x &:= \{y\in \zz^d: g_x(y) > g(x)\},\\
\Delta_x &:= \{y\in Q_x: y \text{ adjacent to } \zz^d\backslash Q_x, \ y \text{ not adjacent to } G_x\}, \\
D_x &:= \{y\in Q_x: y \text{ adjacent to } G_x\}. 
\end{align*}
The following Lemma is simply a slightly altered copy of  Lemma 3.3 in \cite{alexander97}. 
\begin{lmm}\label{copy}
Assume the conditions of Theorem \ref{kpzthm}.  Then there exists a constant $C_2$ such that if $|x|\ge C_2$, the following hold.
\begin{itemize}
\item[(i)] If $y\in Q_x$ then $g(y) \le 2g(x)$ and $|y|\le 2d|x|$. 
\item[(ii)] If $y\in \Delta_x$ then $s_x(y)\ge C_1 |x|^{\chi'}(\log |x|)/2$. 
\item[(iii)] If $y\in D_x$ then $g_x(y) \ge 5g(x)/6$. 
\end{itemize}
\end{lmm}
\begin{proof}
(i) Suppose $g(y) > 2g(x)$ and $g_x(y) \le g(x)$. Then using \eqref{gh} and \eqref{gg}, 
\begin{align*}
2g(x) < g(y) \le h(y) = g_x(y)+s_x(y) \le g(x) + s_x(y), 
\end{align*}
so from \eqref{ge}, $s_x(y) > g(x) > C_1 |x|^{\chi'}\log|x|$ provided $|x| \ge C_2$. Thus $y\not \in Q_x$ and the first conclusion in (i) follows. The second conclusion then follows from~\eqref{ge}. 

(ii) Note that $z= y\pm e_i$ for some $z\in \zz^d \cap Q_x^c \cap G_x^c$ and $i\le d$. From (i) we have $|y|\le 2d|x|$, so $|z|\le (2d+1)|x|$, provided $|x| > 1$. Since $z\not \in Q_x$ we must then have $s_x(z)> C_1 |x|^{\chi'}\log |x|$, while using \eqref{gg},
\[
h(\pm e_i) = s_x(\pm e_i) + g_x(\pm e_i) \ge s_x(\pm e_i) - g(\pm e_i).
\]
Consequently, by \eqref{ss},  if $|x|\ge C_2$, 
\begin{align*}
s_x(y) &\ge s_x(z) - s_x(\pm e_i) \\
&\ge C_1 |x|^{\chi'}\log|x| - h(\pm e_i) - g(\pm e_i) \\
&\ge C_1 |x|^{\chi'}(\log|x|) /2. 
\end{align*}
(iii) As in (ii) we have $z=y\pm e_i$ for some $z\in \zz^d \cap G_x$ and $i\le d$. Therefore using \eqref{ge} and \eqref{gg}, 
\[
g_x(y) = g_x(z) - g_x(\pm e_i) \ge g_x(z)- g(\pm e_i) \ge 5g(x)/6
\]
for all $|x|\ge C_2$. 
\end{proof}
Let us call the $m+1$ sites in Lemma \ref{alexander} marked sites. If $m$ is unrestricted, it is easy to find inductively a sequence of marked sites for any path $\gamma$ from $0$ to $nx$, as follows. One can start at $v_0 = 0$, and given $v_i$, let $v_{i+1}'$ be the first site (if any) in $\gamma$, coming after $v_i$, such that $v_{i+1}' - v_i \not \in Q_x$; then let $v_{i+1}$ be the last site in $\gamma$ before $v_{i+1}'$ if $v_{i+1}'$ exists; otherwise let $v_{i+1}=nx$ and end the construction. If $|x|$ is large enough, then it is easy to deduce from \eqref{ge} and \eqref{gg} that all neighbors of the origin must belong to $Q_x$ and therefore $v_{i+1}\ne v_i$ for each $i$ and hence the construction must end after a finite number of steps. We call the sequence of marked sites obtained from a self-avoiding path $\gamma$ in this way, the $Q_x$-skeleton of $\gamma$. 

Given such a skeleton $(v_0,\ldots, v_m)$, abbreviated $(v_i)$, of some lattice path, we divide the corresponding indices into two classes, corresponding to `long' and `short' increments:
\begin{align*}
S((v_i)) &:= \{i: 0\le i< m-1, \ v_{i+1}-v_i\in \Delta_x\},\\
L((v_i)) &:= \{i: 0\le i< m-1, \ v_{i+1}-v_i\in D_x\}. 
\end{align*}
Note that the final index $m$ is in neither class, and by Lemma \ref{copy}(ii), 
\begin{equation}\label{svi}
j\in S((v_i)) \ \text{ implies } \ s_x(v_{j+1}-v_j) > C_1 |x|^{\chi'}(\log|x|) /2.
\end{equation}
The next result is analogous to Proposition 3.4 in \cite{alexander97}. 
\begin{prop}\label{copy2}
Assume the conditions of Theorem \ref{kpzthm}. There exists a constant $C_3$ such that if $|x|\ge C_3$ then for sufficiently large $n$ there exists a lattice path from $0$ to $nx$ with $Q_x$-skeleton of $2n+1$ or fewer vertices. 
\end{prop}
\begin{proof}
Let $(v_0,\ldots, v_m)$ be a $Q_x$-skeleton of some lattice path and let 
\[
Y_i := \ee T(v_i, v_{i+1}) - T(v_i, v_{i+1}). 
\]
Then by \eqref{up1} of Theorem \ref{kpzthm} and Lemma \ref{copy}(i), there are constants $C_4 := \alpha/(2d)^{\chi'} \ge \alpha/2d$ and $C_5$  such that for $0\le i\le m-1$, 
\begin{equation}\label{yp}
\ee\exp(C_4 |Y_i|/|x|^{\chi'}) \le C_5. 
\end{equation}
Let $Y_0', Y_1',\ldots, Y_{m-1}'$ be independent random variables with $Y_i'$ having the same distribution as $Y_i$. Let $T(0, w; (v_j))$ be the minimum passage time among all lattice paths from $0$ to a site $w$ with $Q_x$-skeleton $(v_j)$. By \cite[equation (4.13)]{kesten86} or \cite[Theorem 2.3]{alexander93}, for all $t\ge 0$,
\[
\pp\biggl(\sum_{i=0}^{m-1} Y_i' \ge t\biggr) \ge \pp\biggl(\sum_{i=0}^{m-1} \ee T(v_i, v_{i+1}) - T(0, v_m; (v_j)) \ge t\biggr). 
\]
Now by \eqref{yp},
\begin{align*}
\pp\biggl(\sum_{i=0}^{m-1} Y_i' \ge t\biggr) &\le  e^{-C_4 t/|x|^{\chi'}} C_5^m. 
\end{align*}
Let $C_6 := 20 d^2/\alpha$. Taking $t = C_6 m |x|^{\chi'}\log|x|$, the above display shows that there is a constant $C_7$ such that for all $|x| \ge C_7$, 
\begin{align*}
\pp\biggl(\sum_{i=0}^{m-1} \ee T(v_i, v_{i+1}) - T(0, v_m; (v_j)) \ge C_6m |x|^{\chi'}\log|x|\biggr) &\le (C_5e^{- 10 d \log |x|})^m.
\end{align*}
From the definition of a $Q_x$-skeleton, it is easy to see that there is a constant $C_8$ such that there are at most $(C_8 |x|^d)^m$ $Q_x$-skeletons with $m+1$ vertices.  Therefore, the above display shows that there are constants $C_9$ and $C_{10}$ such that when $|x| \ge C_9$,
\begin{align*}
&\pp\biggl(\sum_{i=0}^{m-1} \ee T(v_i, v_{i+1}) - T(0, v_m; (v_j)) \ge C_6m |x|^{\chi'}\log|x|\\
&\qquad \qquad \text{ for some $Q_x$-skeleton with $m+1$ vertices}\biggr) \le e^{-C_{10} m \log |x|}.
\end{align*}
This in turn yields that for some constant $C_{11}$, for all $|x|\ge C_{11}$,
\begin{equation}\label{alexmain}
\begin{split}
&\pp\biggl(\sum_{i=0}^{m-1} \ee T(v_i, v_{i+1}) - T(0, v_m; (v_j)) \ge C_6m |x|^{\chi'}\log|x|\\
&\qquad \text{ for some $m \ge 1$ and some $Q_x$-skeleton with $m+1$ vertices}\biggr) \\
&\qquad \qquad \le 2e^{-C_{10} \log |x|}.
\end{split}
\end{equation}
Now let $\omega := \{t_e: e \text{ is an edge in } \zz^d\}$ be a fixed configuration of passage times (to be further specified later) and let $(v_0,\ldots, v_m)$ be the $Q_x$-skeleton of a route from $0$ to $nx$. Then since $v_{i+1}-v_i\in Q_x$, 
\[
m g(x) \ge \sum_{i=0}^{m-1} g_x(v_{i+1}-v_i) = g_x(nx) = n g(x). 
\]
Therefore 
\begin{equation}\label{nm}
n\le m. 
\end{equation}
From the concentration of first-passage times,  
\[
\pp(T(0,nx) \le n g(x)+n) \ra 1 \ \text{ as } n \ra\infty, 
\]
so by \eqref{alexmain} if $n$ is large there exists a configuration $\omega$ and a $Q_x$-skeleton $(v_0,\ldots, v_m)$ of a path from $0$ to $nx$ such that 
\begin{align}\label{ttg}
T(0, nx; (v_j)) = T(0,nx) \le ng(x)+n
\end{align}
and 
\begin{align}\label{ttg1}
\sum_{i=0}^{m-1} \ee T(v_i, v_{i+1}) - T(0, nx; (v_j)) < C_6m |x|^{\chi'}\log|x|. 
\end{align}
Thus for some constant $C_{12}$, if $|x|\ge C_{12}$ then by \eqref{nm}, \eqref{ttg} and \eqref{ttg1},
\begin{equation}\label{tvv}
\begin{split}
\sum_{i=0}^{m-1} \ee T(v_i, v_{i+1}) &< ng(x) + n + C_6m |x|^{\chi'}\log|x|\\
&\le n g(x) + 2C_6m |x|^{\chi'}\log|x|. 
\end{split}
\end{equation}
But by \eqref{svi}, 
\begin{align*}
\sum_{i=0}^{m-1} \ee T(v_i, v_{i+1}) &= \sum_{i=0}^{m-1}(g_x(v_{i+1}-v_i) + s_x(v_{i+1}-v_i)) \\
&\ge g_x(nx) + C_1 |S((v_i))| |x|^{\chi'} (\log|x|)/2,
\end{align*}
which, together with \eqref{tvv}, yields
\begin{equation}\label{sineq}
|S((v_i))| \le 4C_6m/C_1 = m/4. 
\end{equation}
At the same time, using Lemma \ref{copy}(iii), 
\begin{align*}
\sum_{i=0}^{m-1} \ee T(v_i, v_{i+1}) &= \sum_{i=0}^{m-1}(g_x(v_{i+1}-v_i) + s_x(v_{i+1}-v_i)) \\
&\ge 5|L((v_i))| g(x)/6. 
\end{align*}
With \eqref{tvv}, \eqref{ge} and the assumption that $\chi' < 1$, this implies that there is a constant $C_{13}$ such that, provided $|x|\ge C_{13}$, 
\[
|L((v_i))| \le 6n/5 + \frac{12C_6 m|x|^{\chi'}\log|x|}{6g(e_1) |x|/\sqrt{d}} \le 6n/5 + m/8. 
\]
This and \eqref{sineq} give
\[
m = |L((v_i))| + |S((v_i))| + 1\le 6n/5 + 3m/8 + 1,
\]
which, for $n$ large, implies $m\le 2n$, proving the Proposition. 
\end{proof}
\begin{proof}[Proof of Theorem \ref{gapthm}]
Lemma \ref{alexander} and Proposition \ref{copy2} prove the conclusion of Theorem \ref{gapthm} for $x$ with sufficiently large Euclidean norm. To prove this for all $x$ with $|x|>1$, one simply has to increase the value of $C$. 
\end{proof}

\section{Curvature bounds}
The unit ball of the $g$-norm, usually called the `limit shape' of first-passage percolation, is an object of great interest and intrigue in this literature. Very little is known rigorously about the limit shape, except for a fundamental result about convergence to the limit shape due to Cox and Durrett \cite{cd81}, some qualitative results of Kesten \cite{kesten86} who proved, in particular, that the limit shape may not be an Euclidean ball, an important result of Durrett and Liggett \cite{dl81} who showed that the boundary of the limit shape may contain straight lines, and some bounds on the rate of convergence to the limit shape \cite{kesten93, alexander97}. In particular, it is not even known whether the limit shape may be strictly convex in every direction (except for the related continuum model of `Riemannian first-passage percolation' \cite{lw10} and first-passage percolation with stationary ergodic edge-weights \cite{hm95}). 

The following Proposition lists two properties of the limit shape that are crucial for our purposes.
\begin{prop}\label{curve}
Let $g$ be defined as in \eqref{gdef} and assume that the distribution of edge-weights is continuous. Then there exists $x_0\in \rr^d$ with $|x_0| =1$, a constant $C\ge 0$ and a hyperplane $H_0$ through the origin perpendicular to $x_0$ such that for all $z\in H_0$,
\[
|g(x_0+z) - g(x_0)| \le C|z|^2. 
\]
There also exists $x_1\in \rr^d$ with $|x_1|=1$ and a hyperplane $H_1$ through the origin perpendicular to $x_1$ such that for all $z\in H_1$,
\[
g(x_1+z) \ge \sqrt{ 1+|z|^2}  g(x_1). 
\]
\end{prop}
\begin{proof}
The proof is similar to that of \cite[Lemma 5]{np95}. 
Let $B(0,r)$ denote the Euclidean ball of radius $r$ centered at the origin and let 
\[
B_g(0,r) := \{x: g(x)\le r\}
\]
denote the ball of radius $r$ centered at the origin for the norm $g$. Let $r$ be the smallest number such that $B_g(0,r)\supseteq B(0,1)$. Let $x_0$ be a point of intersection of $\partial B_g(0,r)$ and $\partial B(0,1)$. Let $H_0$ be a hyperplane tangent to $\partial B_g(0,r)$ at $x_0$, translated to contain the origin. Note that $x_0+H_0$ is also a tangent hyperplane for $B(0,1)$ at $x_0$, since it touches $B(0,1)$ only at~$x_0$. Therefore $H_0$ is perpendicular to $x_0$. Now for any $z\in H_0$,  the point $y := (x_0+z)/|x_0+z|$ is a point on $\partial B(0,1)$ and hence contained in $B_g(0,r)$. Therefore 
\[
g(x_0)=r \ge g(y) = \frac{1}{|x_0+z|} g(x_0+z)= \frac{1}{\sqrt{1+|z|^2}} g(x_0+z). 
\]
Since $g(x_0+z)$ grows like $|z|$ as $|z|\ra\infty$, this shows that there is a constant $C$ such that 
\[
g(x_0+z) \le g(x_0)+C|z|^2
\]
for all $z\in H_0$. Also, since $x_0+z\not\in B_g(0,r)$ for $z\in H_0\backslash\{0\}$, therefore $g(x_0)\le g(x_0+z)$ for all $z\in H_0$. This proves the first assertion of the Proposition.

For the second, we proceed similarly. Let $r$ be the largest number such that $B_g(0,r) \subseteq B(0,1)$.  Let $x_1$ be a point in the intersection of $\partial B_g(0,r)$ and $\partial B(0,1)$. Let $H_1$ be the hyperplane tangent to $\partial B(0,1)$ at $x_1$, translated to contain the origin. Note that this is simply the hyperplane through the origin that is perpendicular to $x_1$. Since $B(0,1)$ contains $B_g(0,r)$, and $y:= (x_1+z)/|x_1+z|$ is a point in $\partial B(0,1)$, therefore
\[
g(x_1) =r \le g(y) = \frac{1}{|x_1+z|} g(x_1 + z) = \frac{1}{\sqrt{1 + |z|^2}} g(x_1+z). 
\]
This completes the argument.
\end{proof}

\section{Proof of $\chi_a \ge 2\xi_b -1 $}\label{easypart}
We will prove by contradiction. Suppose that $2\xi_b-1 > \chi_a$. Choose $\xi'$ such that
\begin{equation*}\label{cxx}
\frac{1+\chi_a}{2} < \xi' < \xi_b.
\end{equation*}
Note that $\xi'< 1$. Let $x_1$ and $H_1$ be  as in Proposition \ref{curve}. Let $n$ be a positive integer, to be chosen later. Throughout this proof, $C$ will denote any positive constant that does not depend on $n$. The value of $C$ may change from line to line. Also, we will assume without mention that `$n$ is large enough' wherever required. 

Let $y$ be the closest point in $\zz^d$ to $nx_1$. Note that 
\begin{equation}\label{ynx}
|y-nx_1|\le \sqrt{d}.
\end{equation}
Let $L$ denote the line passing through $0$ and $nx_1$ and let $L'$ denote the line segment joining $0$ to $nx_1$ (but not including the endpoints). Let $V$ be the set of all points in $\zz^d$ whose distance from $L'$ lies in the interval $[n^{\xi'}, 2n^{\xi'}]$. Take any $v\in V$.  We claim that there is a constant $C$ (not depending on $n$) such that for any $v\in V$,
\begin{align}\label{ggg}
g(v)+ g(nx_1-v) \ge g(nx_1) + C n^{2\xi'-1}. 
\end{align}
Let us now prove this claim. Let $w$ be the projection of $v$ onto $L$ along $H_1$ (i.e.\ the perpendicular projection). To prove \eqref{ggg}, there are three cases to consider. First suppose that $w$ lies in $L'$. Note that $w/|w| = x_1$. Let $v' := v/|w|$ and $z := v'- x_1 = (v-w)/|w|$.  
\begin{figure}[h]
\begin{pdfpic}
\begin{pspicture}(-.5,-.5)(10,3)
\psset{xunit=1cm,yunit=1cm}
\rput(3.2, .15){\small $w$}
\rput(0,.2){\small $0$}
\rput(9,.2){\small $nx_1$}
\psline{*-*}(0,0)(9,0)
\rput(3.2, 2.4){\small $v$}
\psline[linestyle=dashed]{*-*}(3, 2.5)(3,0)
\psline[linestyle=dashed]{-}(3,2.5)(0,0)
\psline[linestyle=dashed]{*-*}(.75, .625)(.75,0)
\rput(.75, .9){\small $v'$}
\rput(.95,.15){\small $x_1$}
\end{pspicture}
\end{pdfpic}
\caption{The relative positions of $x_1, v', v, w, nx_1$. }
\label{fig}
\end{figure}
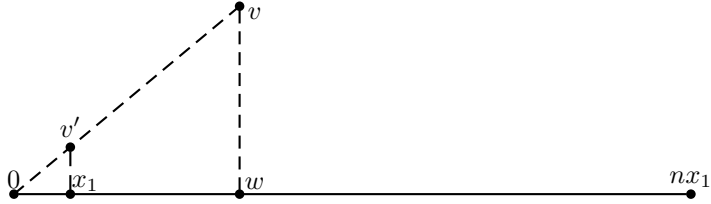

Note that $z\in H_1$. Thus by Proposition \ref{curve}, 
\[
g(v') = g (x_1 + z) \ge \sqrt{1+|z|^2} g(x_1).
\]
Consequently, 
\begin{equation}\label{gv1}
g(v) \ge |w|\sqrt{1+|z|^2} g(x_1). 
\end{equation}
Next, let $w' := nx_1 - w$. Note that $w'/|w'| = x_1$. let $v'' := (nx_1 - v)/|w'|$, and
\[
z' := v'' - x_1 = (w-v)/|w'|. 
\]
Then $z' \in H_1$, and hence by Proposition \ref{curve},
\[
g(v'') = g(x_1 + z') \ge \sqrt{1+|z'|^2} g(x_1). 
\]
Consequently,
\begin{equation}\label{gv2}
g(nx_1 - v) \ge |w'|\sqrt{1+|z'|^2}g(x_1). 
\end{equation}
Since $v\in V$, therefore $|v-w| \ge n^{\xi'}$. Again, $|w|+|w'| = n$. Thus,
\[
\min\{|z|, |z'|\} \ge n^{\xi'-1}.
\] 
Combining this with \eqref{gv1}, \eqref{gv2}, \eqref{ge}  and the fact that $\xi' < 1$, we have  
\begin{align*}
g(v) + g(nx_1 - v) &\ge (|w|+ |w'|) \sqrt{1+n^{2\xi'-2}}g(x_1)\\
&= \sqrt{1+n^{2\xi'-2}} g(nx_1)\\
&\ge g(nx_1) + Cn^{2\xi' -1}.
\end{align*}
Next, suppose that $w$ lies in $L\backslash L'$, on the side closer to $nx_1$. As above, let $z := (v-w)/|w|$. As in \eqref{gv1}, we conclude that
\begin{equation}\label{gvw}
g(v) \ge |w| \sqrt{1+|z|^2} g(x_1). 
\end{equation}
By the definition of $V$, the distance between $v$ and $nx_1$ must be greater than $n^{\xi'}$. But in this case 
\[
|v-nx_1|^2 = (|w|-n)^2 + |v-w|^2 = (|w|-n)^2 + |w|^2|z|^2,
\]
and we also have $n\le |w| \le 3n$. Thus, either $|w|^2|z|^2 > n^{2\xi'}/2$ (which implies $|z|^2\ge Cn^{2\xi'-2}$), or $|w| \ge n + n^{\xi'}/\sqrt{2}$. Since $\xi' > 2\xi'-1$, therefore by \eqref{gvw}, in either situation we have 
\[
g(v)\ge g(nx_1) + Cn^{2\xi' -1}. 
\]
Similarly, if $w$ lies in $L\backslash L'$, on the side closer to $0$, then 
\[
g(nx_1 - v) \ge g(nx_1) + Cn^{2\xi'-1}. 
\]
This completes the proof of \eqref{ggg}. Now \eqref{ggg} combined with Theorem \ref{gapthm},~\eqref{ynx} and the fact that $2\xi'-1 > \chi_a$ implies that if $n$ is large enough, then for any $v\in V$, 
\begin{equation}\label{hhh}
h(v) + h(y-v) \ge h(y) + Cn^{2\xi'-1}.
\end{equation}
Choose $\chi_1, \chi_2$ such that $\chi_a < \chi_1< \chi_2 < 2\xi' -1$. Then by \eqref{up1} of Theorem~\ref{kpzthm}, there is a constant $C$ such that for $n$ large enough,
\[
\pp(T(0,y) > h(y) + n^{\chi_2}) \le e^{-Cn^{\chi_2-\chi_1}}.
\]
Now, for any $v\in V$, both $|v|$ and $|y-v|$ are bounded above by $Cn$. Therefore again by \eqref{up1}, 
\begin{align*}
\pp(T(0,v) < h(v) - n^{\chi_2}) &\le e^{-Cn^{\chi_2-\chi_1}},\\
\pp(T(v,y) < h(y-v) - n^{\chi_2}) &\le e^{-Cn^{\chi_2-\chi_1}}. 
\end{align*}
This, together with \eqref{hhh}, shows that if $n$ is large enough, then for any $v\in V$,
\[
\pp(T(0,y) = T(0,v) + T(v,y)) \le e^{-C n^{\chi_2-\chi_1}}.
\] 
Since the size of $V$ grows polynomially with $n$, this shows that 
\[
\pp(T(0,y) = T(0,v) + T(v,y) \text{ for some } v\in V) \le e^{-C n^{\chi_2-\chi_1}}.
\]
Note that if the geodesic from $0$ to $y$ passes through $V$, then $T(0,y) = T(0,v) + T(v,y)$ for some $v\in V$. If $D(0,y) > n^{\xi'}$ then the geodesic must pass through $V$. Thus, the above inequality implies that
\[
\pp(D(0,y)> n^{\xi'}) \le e^{-C n^{\chi_2-\chi_1}}.
\]
By \eqref{up2} of Theorem \ref{kpzthm}, this gives
\begin{align*}
\ee D(0,y) &\le n^{\xi'} + \ee(D(0,y) 1_{\{D(0,y) >  n^{\xi'}\}})\\
&\le n^{\xi'} + \sqrt{\ee (D(0,y)^2) \pp(D(0,y) > n^{\xi'})}\\
&\le n^{\xi'} + C_1n^{C_1} e^{-C_2 n^{\chi_2-\chi_1}}. 
\end{align*}
Taking $n \ra\infty$, this shows that \eqref{down2} of Theorem \ref{kpzthm} is violated (since $\xi' < \xi_b$), leading to a contradiction to our original assumption that $\chi_a < 2\xi_b -1$. Thus, $\chi_a \ge 2\xi_b -1$.

\section{Proof of $\chi \le 2\xi -1$ when $0 < \chi < 1/2$}\label{012}
In this section and the rest of the manuscript, we assume that $\chi_a=\chi_b$ and $\xi_a = \xi_b$, and denote these two numbers by $\chi$ and $\xi$.

Again we prove by contradiction. Suppose that $0 < \chi < 1/2$ and $\chi > 2\xi -1$. Fix $\chi_1< \chi < \chi_2$, to be chosen later. 
Choose $\xi'$ such that
\[
\xi <  \xi' < \frac{1+\chi}{2}. 
\]
Define:
\begin{align*}
\beta' &:= \frac{1}{2} + \frac{\xi'}{1+\chi}. \\
\beta &:= 1-\frac{\chi}{2} + \frac{\chi}{2}\beta'. \\
\ve &:= (1-\beta)\biggl(1-\frac{\chi}{2}\biggr). 
\end{align*}
We need several inequalities involving the numbers $\beta'$, $\beta$ and $\ve$. Since 
\[
0 < \frac{\xi'}{1+\chi} < \frac{1}{2},
\]
therefore
\begin{equation}\label{b1}
\frac{1}{2}< \beta' < 1.
\end{equation}
Since $\chi < 1$ and $\xi'<(1+\chi)/2 <1$, 
\begin{equation}\label{bpxi}
\beta' > \frac{1}{2} + \frac{\xi'}{2} > \xi'. 
\end{equation}
Since $\beta$ is a convex combination of $1$ and $\beta'$ and $\chi > 0$, 
\begin{equation}\label{bb}
\beta' < \beta < 1. 
\end{equation}
Since $0 < \chi < 1$ and $0<\beta<1$,
\begin{equation}\label{veb}
0 <\ve < 1-\beta. 
\end{equation}
Since $\beta'$ is the average of $1$ and $2\xi'/(1+\chi)\in (0,1)$, therefore $\beta'$ is strictly bigger than $2\xi'/(1+\chi)$ and hence 
\begin{equation}\label{xibp}
\begin{split}
2\xi' - \beta' &< 2\xi' - \frac{2\xi'}{1+\chi}\\
&= \frac{2\xi'}{1+\chi} \chi < \beta' \chi. 
\end{split}
\end{equation}
By \eqref{bb}, this implies that 
\begin{equation}\label{xib}
2\xi' - \beta < 2\xi' - \beta' < \beta' \chi < \beta \chi. 
\end{equation}
Next, by \eqref{b1}, 
\begin{equation}\label{chib}
\begin{split}
1-\beta + \beta'\chi = \frac{\chi}{2}(1+\beta') < \chi. 
\end{split}
\end{equation}
And finally by \eqref{b1},
\begin{equation}\label{chibep}
\beta \chi + 1-\beta - \ve = \beta \chi + (1-\beta) \frac{\chi}{2} < \chi. 
\end{equation}
Let $q$ be a large positive integer, to be chosen later. Throughout this proof, we will assume without mention  that $q$ is `large enough' wherever required. Also, $C$ will denote any constant that does not depend on our choice of $q$, but may depend on all other parameters.

Let $r$ be an integer between $\frac{1}{2}q^{(1-\beta-\ve)/\ve}$ and  $2q^{(1-\beta-\ve)/\ve}$, recalling that by~\eqref{veb}, $1-\beta-\ve > 0$. Let $k= rq$. Let $a$ be a real number between $q^{\beta/\ve}$ and $2q^{\beta/\ve}$. Let $n = ak$. Note that $n= arq$, which gives $\frac{1}{2}q^{1/\ve} \le n\le 4q^{1/\ve}$. From this it is easy to see that there are positive constants $C_1$ and $C_2$, depending only on $\beta$ and $\ve$, such that
\begin{align}
&C_1n^{\ve} \le q \le C_2n^{\ve}, \label{qbd}\\
&C_1n^{1-\beta} \le k\le C_2n^{1-\beta}, \label{kbd}\\
&C_1n^\beta \le a \le C_2n^\beta, \label{abd}\\ 
&C_1n^{1-\beta-\ve} < r < C_2n^{1-\beta-\ve}. \label{rbd}
\end{align}
Let $b := n^{\beta'}$. Note that by \eqref{bb}, $b$ is negligible compared to $a$ if $q$ is large. Note also that, although $r$, $k$ and $q$ are integers, $a$, $n$ and $b$ need not be. 

Let $x_0$ and $H_0$ be as in Proposition~\ref{curve}. For $0\le i\le k$, define
\begin{align*}
U_i' &:= H_0 + i a x_0\ , \\ 
V_i' &:= H_0 + (ia + a-b) x_0\ . 
\end{align*}
Let $U_i$ be the set of points in $\zz^d$ that are within distance $\sqrt{d}$ from $U_i'$. Let $V_i$ be the set of points in $\zz^d$ that are within distance $\sqrt{d}$ from $V_i'$.

For $0\le i\le k$ let $y_i$ be the closest point in $\zz^d$ to $iax_0$, and let $z_i$ be the closest point in $\zz^d$ to $(ia + a-b) x_0$, applying some arbitrary rule to break ties. Note that if $x\in \rr^d$, and $y\in\zz^d$ is closest to $x$, then $|x-y|\le \sqrt{d}$. Therefore $y_i\in U_i$ and $z_i\in V_i$. Figure \ref{fignew} gives a pictorial representation of the above definitions, assuming for simplicity that $U_i=U_i'$ and $V_i=V_i'$.

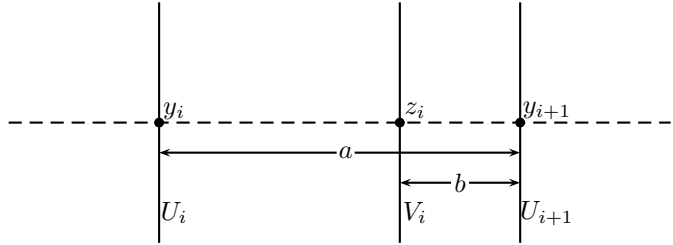
\begin{figure}[h]
\begin{pdfpic}
\begin{pspicture}(-1,-1)(13,3)
\psset{xunit=.8cm,yunit=.8cm}
\psline{-}(4,-1)(4,3)
\psline{-}(8,-1)(8,3)
\psline{-}(10, -1)(10,3)
\psline[linestyle=dashed]{-}(1.5,1)(12.5,1)
\psline{<-}(4,.5)(6.95, .5)
\psline{->}(7.25, .5)(10,.5)
\rput(7.1, .5){\small $a$}
\psline{<-}(8, 0)(8.85,0)
\rput(9,0){\small $b$}
\psline{->}(9.15, 0)(10,0)
\psline{*-*}(4,1)(4,1)
\psline{*-*}(8,1)(8,1)
\psline{*-*}(10,1)(10,1)
\rput(4.25, 1.2){\small $y_i$}
\rput(8.25, 1.2){\small $z_i$}
\rput(10.45, 1.2){\small $y_{i+1}$}
\rput(4.25, -.5){\small $U_i$}
\rput(8.25, -.5){\small $V_i$}
\rput(10.45, -.5){\small $U_{i+1}$}
\end{pspicture}
\end{pdfpic}
\caption{Diagrammatic representation of $y_i$, $z_i$, $U_i$ and $V_i$.}
\label{fignew}
\end{figure}

Let $U_i^o$ be the subset of $U_i$ that is within distance $n^{\xi'}$ from $y_i$. Similarly let $V_i^o$ be the subset of $V_i$ that is within distance $n^{\xi'}$ from $z_i$. 

For any $A, B\subseteq \zz^d$, let $T(A,B)$ denote the minimum passage time from $A$ to $B$.  Let $G(A, B)$ denote the (unique) geodesic from $A$ to $B$, so that $T(A, B)$ is the sum of edge-weights of  $G(A,B)$. 

Fix any two integers $0\le l<m\le k$ such that $m-l > 3$. Consider the geodesic $G := G(y_l, y_m)$. Since $x_0\not \in H_0$, it is easy to see that $G$ must `hit' each $U_i$ and $V_i$, $l\le i\le m-1$. Arranging the vertices of $G$ in a sequence starting at $y_l$ and ending at $y_m$, for each $l\le i< m$ let $u_i'$ be the first vertex in $U_i$ visited by $G$ and let $v_i'$ be the first vertex in $V_i$ visited by $G$. Let $u_m' := y_m$. Note that $G$ visits these vertices in the order $u_l', v_l', u_{l+1}', v_{l+1}', \ldots, v_{m-1}', u_m'$.  Figure \ref{figg} gives a pictorial representation of the points $u_i'$ and $v_i'$ on the geodesic~$G$. 
\begin{figure}[h]
\begin{pdfpic}
\begin{pspicture}(-1,-1)(10,3)
\psset{xunit=1cm,yunit=1cm}
\psline{-}(0, 0)(0, 3)
\psline{-}(0,0)(6.5,0)
\psline{-}(0,3)(6.5,3)
\psline[linestyle=dashed]{-}(6.5,0)(8,0)
\psline[linestyle=dashed]{-}(6.5,3)(8,3)
\pspolygon[fillstyle=solid, fillcolor=lightgray](4,0)(4,3)(6,3)(6,0)
\pscurve{-}(0,1.5)(.2,1.7)(-.2, 1.9)(1,2.6)(3.5,1.6)(4,1.9)(4.8,1.5)(3.7,.8)(3.65,.7)(5, .6)(5.8, 1.1)(6,1.2)(6.1,1.4)(5.7, 1.6)(6, 2)(6.5, 2.2)
\pscurve[linestyle=dashed]{-}(6.5,2.2)(7, 2.3)(7.5,1.8)(8,1.9)
\pscurve{*}(0,1.5)(0,1.5)(0,1.5)
\pscurve{*}(4,1.9)(4,1.9)(4,1.9)
\pscurve{*}(6,1.2)(6,1.2)(6,1.2)
\rput(1.9, 2.6){\small $G$}
\rput(0.3,1.3){\small $u_0'$}
\rput(4.25,1.65){\small $v_0'$}
\rput(6.25, 1.1){\small $u_1'$}
\rput(2,-.19){\small $n^\beta$}
\psline{<-}(0.05,-.3)(1.75, -.3)
\psline{->}(2.15,-.3)(3.95,-.3)
\rput(5,-.17){\small $n^{\beta'}$}
\psline{<-}(4.05,-.3)(4.75, -.3)
\psline{->}(5.15,-.3)(5.95,-.3)
\end{pspicture}
\end{pdfpic}
\caption{Location of $u_0', v_0', u_1', v_1', \ldots$ on the geodesic $G$.}
\label{figg}
\end{figure}
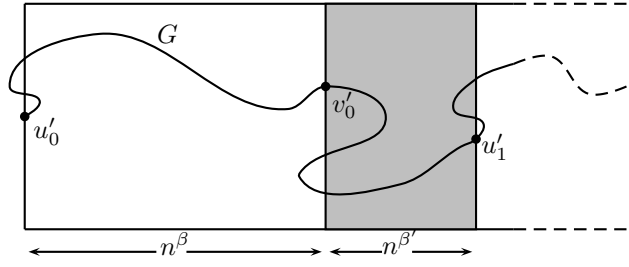
Let $T_i'$ be the sum of edge-weights of the portion of $G$ from $u_i'$ to $v_i'$. Let $E$ be the event that $u_i'\in U_i^o$ and $v_i'\in V_i^o$ for each $i$. If $E$ happens, then clearly 
\[
T_i' \ge T(U_i^o, V_i^o). 
\]
Similarly, note that weight of the part of $G$ from $v_i'$ to $u_{i+1}'$ must exceed or equal $T(v_i', u_{i+1}')$. Therefore, if $E$ happens, then 
\begin{equation}\label{lower}
\begin{split}
T(y_{l}, y_{m}) &\ge \sum_{i=l}^{m-1} T_i' + \sum_{i=l}^{m-1} T(v_i', u_{i+1}')\\
&\ge \sum_{i=l}^{m-1} T(U_i^o, V_i^o) + \sum_{i=l}^{m-1} T(v_i', u_{i+1}'). 
\end{split}
\end{equation}
Next, for each $0\le i< k$ let $G_i := G(U_i^o, V_i^o)$. Let $u_i$ and $v_i$ be the endpoints of $G_i$. Let $G_i' := G(v_i, u_{i+1})$. Figure~\ref{figgg} gives a picture illustrating  the paths $G_i$ and $G_i'$.
\begin{figure}[h]
\begin{pdfpic}
\begin{pspicture}(-1,-1)(10,3)
\psset{xunit=1cm,yunit=1cm}
\pspolygon[fillstyle=solid](0,0)(0,2)(9,2)(9,0)
\pspolygon[fillstyle=solid, fillcolor=lightgray](3,0)(3,2)(2,2)(2,0)
\pspolygon[fillstyle=solid, fillcolor=lightgray](6,0)(6,2)(5,2)(5,0)
\pspolygon[fillstyle=solid, fillcolor=lightgray](9,0)(9,2)(8,2)(8,0)
\psline{-}(9,2)(9.5,2)
\psline[linestyle=dashed]{-}(9.5,2)(10.5,2)
\psline{-}(9,0)(9.5,0)
\psline[linestyle=dashed]{-}(9.5,0)(10.5,0)
\pscurve{-}(9,1)(9.3,1.3)(9.5, 1.2)
\pscurve[linestyle=dashed](9.5,1.2)(9.7,1.1)(10.1,1.3)(10.5,1.3)
\pscurve{*-*}(0,1.5)(1, 1.7)(1.5,.7)(2,1)
\pscurve[linestyle=dashed]{*-*}(2,1,1)(2.2, 1.1)(2.7,1.3)(3,1.5)
\pscurve{*-*}(3,1.5)(4, 1.3)(4.5,1.6)(5,1.3)
\pscurve[linestyle=dashed]{*-*}(5, 1.3)(5.5,1.5)(6,1.4)
\pscurve{*-*}(6,1.4)(6.5,1.3)(7, 1.7)(7.5,.7)(8,.8)
\pscurve[linestyle=dashed]{*-*}(8, .8)(8.3, 1.3)(8.6,.9)(9,1)
\rput(.23,1.33){\small $u_0$}
\rput(2.23,.83){\small $v_0$}
\rput(3.22, 1.24){\small $u_1$}
\rput(5.23, 1.13){\small $v_1$}
\rput(6.23, 1.23){\small $u_2$}
\rput(8.22, .63){\small $v_2$}
\rput(9.22, .83){\small $u_3$}
\rput(1.1, .9){\small $G_0$}
\rput(2.4, 1.45){\small $G_0'$}
\rput(4, 1.12){\small $G_1$}
\rput(5.4, 1.67){\small $G_1'$}
\rput(7.1, .9){\small $G_2$}
\rput(8.3, 1.47){\small $G_2'$}
\rput(1,-.2){\small $n^\beta$}
\psline{<-}(0.05,-.3)(.7, -.3)
\psline{->}(1.1,-.3)(1.95,-.3)
\rput(2.6,-.2){\small $n^{\beta'}$}
\psline{<-}(2.05,-.3)(2.35, -.3)
\psline{->}(2.7,-.3)(2.95,-.3)
\end{pspicture}
\end{pdfpic}
\caption{The paths $G_0, G_0', G_1, G_1',\ldots$.}
\label{figgg}
\end{figure}
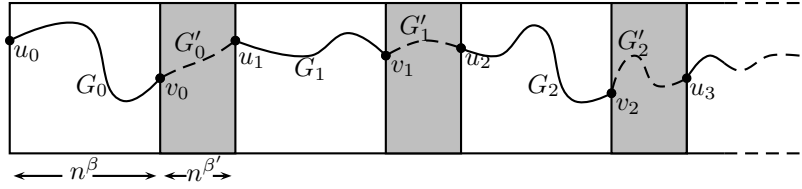
The concatenation of the paths $G(y_{l}, v_{l})$, $G_l'$, $G_{l+1}$, $G_{l+1}'$, $G_{l+2}$, $\ldots$, $G_{m-2}'$, $G_{m-1}$, $G(v_{m-1}, y_{m})$ is a path from $y_{l}$ to $y_{m}$ (need not be self-avoiding). Therefore,
\begin{equation}\label{upper}
\begin{split}
T(y_{l}, y_{m}) &\le T(y_{l}, v_{l})  + \sum_{i=l+1}^{m-1} T(U_i^o, V_i^o) + \sum_{i=l}^{m-2} T(v_i, u_{i+1})\\
&\qquad +  T(v_{m-1}, y_{m}). 
\end{split}
\end{equation}
Define 
\[
\Delta_{l,m} := T(y_l, y_m) -\sum_{i=l}^{m-1} (T(U_i^o, V_i^o) + T(V_i^o, U_{i+1}^o)). 
\]
Combining \eqref{lower} and \eqref{upper} implies that if $E$ happens, then  
\begin{equation*}\label{big}
\begin{split}
|\Delta_{l,m}| &\le \sum_{i=l}^{m-1} |T(V_i^o, U_{i+1}^o) - T(v_i', u_{i+1}')|  + \sum_{i=l}^{m-2} |T(V_i^o, U_{i+1}^o) - T(v_i, u_{i+1})| \\
&\qquad + |T(U_l^o, V_l^o) - T(y_{l}, v_{l})|  + |T(V_{m-1}^o, U_{m}^o)- T(v_{m-1}, y_{m})|.
\end{split}
\end{equation*}
Thus, if 
\begin{align*}
M_i := \max_{v,v'\in V_i^o, \ u,u'\in U_{i+1}^o} |T(v,u) - T(v', u')|,\\
N_i :=  \max_{u,u'\in U_i^o, \ v,v'\in V_{i}^o} |T(u,v) - T(u', v')|,
\end{align*}
and the event $E$ happens, then 
\begin{equation}\label{big2}
\begin{split}
|\Delta_{l,m}| &\le 2\sum_{i=l}^{m-1} M_i + N_l.
\end{split}
\end{equation}
For a random variable $X$, let $\|X\|_p := (\ee|X|^p)^{1/p}$ denote its $L^p$ norm. It is easy to see that  $\|\Delta_{l,m}\|_4\le n^C$, where recall that $C$ stands for any constant that does not depend on our choice of the integer $q$, but may depend on $\chi$, $\xi$, $\xi'$ and the distribution of edge weights. Take any $\xi_1 \in (\xi, \xi')$. By \eqref{up2} of Theorem~\ref{kpzthm}, $\pp(E^c) \le e^{-Cn^{\xi'-\xi_1}}$. Together with \eqref{big2}, this shows that for some constants $C_3$ and $C_4$, 
\begin{equation}\label{big3}
\begin{split}
\|\Delta_{l,m}\|_2 &\le \|\Delta_{l,m} 1_{E^c}\|_2 + \|\Delta_{l,m} 1_E\|_2\\
&\le \|\Delta_{l,m}\|_4 (\pp(E^c))^{1/4} + \|\Delta_{l,m} 1_E\|_2\\
&\le n^{C_3} e^{-C_4n^{\xi'-\xi_1}} + 2\sum_{i=l}^{m-1} \|M_i\|_2 + \|N_l\|_2. 
\end{split}
\end{equation}
Fix $0\le i\le k-1$ and $v\in V_i^o$, $u\in U_{i+1}^o$. Let $x$ be the nearest point to $v$ in $V_i'$ and $y$ be the nearest point to $u$ in $U_{i+1}'$. Then by definition of $V_i'$ and $U_{i+1}'$, there are vectors $z, z'\in H_0$ such that $|z|$ and $|z'|$ are bounded by $Cn^{\xi'}$, and $x = (ia+a-b) x_0+ z$ and $y = (ia+a)x_0 + z'$. Thus by Proposition \ref{curve},
\begin{align*}
|g(y-x) - g(bx_0)| &= |g(b x_0 + z'-z) - g(bx_0)| \\
&= b |g(x_0 + (z'-z)/b) - g(x_0)|\\
&\le \frac{C|z'-z|^2}{b} \le Cn^{2\xi' - \beta'}. 
\end{align*}
Thus, for any $v,v'\in V_i^o$ and $u,u'\in U_{i+1}^o$, 
\[
|g(u-v)-g(u'-v')| \le Cn^{2\xi'-\beta'}.
\]
Note also that $|y-x|\le C(n^{\beta'} + n^{\xi'}) \le Cn^{\beta'}$ by \eqref{bpxi}.  This, together with Theorem \ref{gapthm}, shows that for any $v,v'\in V_i^o$ and $u,u'\in U_{i+1}^o$, 
\begin{equation*}
|\ee T(v,u) - \ee T(v',u')| \le Cn^{2\xi' - \beta'} + Cn^{\beta' \chi_2}\log n. 
\end{equation*}
By \eqref{xibp}, this implies
\begin{equation}\label{tvu}
|\ee T(v,u) - \ee T(v',u')| \le Cn^{\beta' \chi_2}\log n.
\end{equation}
Let
\[
M := \max_{v\in V_i^o, \ u\in U_{i+1}^o} \frac{|T(v,u)- \ee T(v,u)|}{|u-v|^{\chi_2}}. 
\]
By \eqref{up1} of Theorem \ref{kpzthm},
\begin{align*}
\ee (e^{\alpha M}) &\le \sum_{v\in V_i^o, \ u\in U_{i+1}^o} \ee\exp\biggl(\alpha \frac{|T(v,u)- \ee T(v,u)|}{|u-v|^{\chi_2}}\biggr)\\
&\le C |V_i^o| |U_{i+1}^o|\le Cn^C.  
\end{align*}
This implies that $\pp(M > t) \le Cn^C e^{-\alpha t}$, which in turn gives $\|M\|_2\le C \log n$. Let
\[
M' := \max_{v\in V_i^o, \ u\in U_{i+1}^o} |T(v,u)- \ee T(v,u)|.
\]
Since by \eqref{bpxi}, $|u-v| \le C(n^{\beta'}+n^{\xi'})\le Cn^{\beta'}$ for all $v\in V_i^o$, $u\in U_{i+1}^o$, therefore $M' \le Cn^{\beta'\chi_2} M$. Thus,
\[
\|M'\|_2\le Cn^{\beta' \chi_2}\log n. 
\]
From this and \eqref{tvu} it follows that
\[
\|M_i\|_2 \le Cn^{\beta' \chi_2} \log n. 
\]
By an exactly similar sequence of steps, replacing $\beta'$ by $\beta$ everywhere and using \eqref{xib} instead of \eqref{xibp}, one can deduce that 
\[
\|N_i\|_2\le Cn^{\beta \chi_2}\log n. 
\]
Combining with \eqref{big3} this gives
\begin{equation}\label{delta}
\|\Delta_{l,m}\|_2 \le Cn^{\beta\chi_2} \log n + C (m-l) n^{\beta'\chi_2} \log n, 
\end{equation}
since the exponential term in \eqref{big3} is negligible compared to the rest.

Now, from the definition of $\Delta_{l,m}$, the fact that $k=rq$, and the triangle inequality, it is easy to see that 
\begin{align*}
\biggl| T(y_0, y_k) - \sum_{j=0}^{r-1} T(y_{jq}, y_{(j+1)q})\biggr| &\le |\Delta_{0,k}| + \sum_{j=0}^{r-1} |\Delta_{jq, (j+1)q}|. 
\end{align*}
Thus by \eqref{delta}, \eqref{rbd} and \eqref{kbd}, 
\begin{equation}\label{tt}
\begin{split}
\biggl\| T(y_0, y_k) &- \sum_{j=0}^{r-1} T(y_{jq}, y_{(j+1)q})\biggr\|_2 \le \|\Delta_{0,k}\|_2 + \sum_{j=0}^{r-1} \|\Delta_{jq, (j+1)q}\|_2\\
&\le C (r+1) n^{\beta\chi_2} \log n + Ck n^{\beta' \chi_2} \log n\\
&\le Cn^{1-\beta-\ve + \beta\chi_2}\log n + C n^{1-\beta + \beta' \chi_2}\log n.   
\end{split}
\end{equation}
For any two random variables $X$ and $Y$, 
\begin{align}
\bigl|\sqrt{\var(X)} - \sqrt{\var(Y)}\bigr| &= |\|X- \ee X\|_2 - \|Y-\ee Y\|_2| \nonumber \\
&\le \|(X- \ee X) - (Y-\ee Y)\|_2 \nonumber \\
&\le \|X-Y\|_2 + |\ee X-\ee Y| \le 2\|X-Y\|_2. \label{referee1}
\end{align}
Therefore it follows from \eqref{tt} that 
\begin{equation}\label{imp1}
\begin{split}
\biggl|(\var T(y_0, y_k))^{1/2} &- \biggl( \var \sum_{j=0}^{r-1} T(y_{jq}, y_{(j+1)q})\biggr)^{1/2}\biggr| \\
&\le Cn^{1-\beta-\ve + \beta\chi_2}\log n + C n^{1-\beta + \beta' \chi_2}\log n. 
\end{split}
\end{equation}
For any $x,y\in \zz^d$, $T(x,y)$ is an increasing function of the edge weights. So by the Harris-FKG  inequality \cite{harris60}, $\cov(T(x,y), T(x', y')) \ge 0$ for any $x,y,x',y'\in \zz^d$. Therefore by \eqref{down1} of Theorem \ref{kpzthm} and \eqref{abd}, \eqref{rbd} and \eqref{qbd}, 
\begin{equation}\label{imp2}
\begin{split}
\var \sum_{j=0}^{r-1} T(y_{jq}, y_{(j+1)q}) &\ge \sum_{j=0}^{r-1}\var  T(y_{jq}, y_{(j+1)q})\\
&\ge C \sum_{j=0}^{r-1} |y_{jq}- y_{(j+1)q}|^{2\chi_1}\\
&\ge C r (aq)^{2\chi_1} \ge C n^{(1-\beta-\ve) + (\beta+\ve) 2\chi_1}. 
\end{split}
\end{equation}
By the inequalities \eqref{chib} and \eqref{chibep}, we see that if $\chi_1$ and $\chi_2$ are chosen sufficiently close to $\chi$, then $\chi_1$ is strictly bigger than both $1-\beta-\ve +\beta \chi_2$ and $1-\beta+\beta'\chi_2$. Therefore by \eqref{imp1} and \eqref{imp2},  and since $1-\beta - \ve + (\beta+\ve)2\chi_1 > 2\chi_1$, 
\[
\var T(y_0, y_k) \ge C n^{(1-\beta-\ve) + (\beta+\ve) 2\chi_1}.
\]
By \eqref{veb} and the assumption that $\chi < 1/2$, we again have that if $\chi_1$ is chosen sufficiently close to $\chi$,
\[
(1-\beta-\ve) + (\beta+\ve) 2\chi_1 > 2\chi. 
\]
Since $|y_0-y_k|\le Cak \le Cn$ by \eqref{abd} and \eqref{kbd}, therefore taking $q\ra\infty$ (and hence $n\ra\infty$) gives a contradiction to \eqref{up1} of Theorem \ref{kpzthm}, thereby proving that $\chi \le 2\xi -1 $ when $0 < \chi < 1/2$.

\section{Proof of $\chi \le 2\xi -1$ when $\chi = 1/2$}\label{hard2}
Suppose that $\chi=1/2$ and $\chi > 2\xi -1$. Define $\chi_1$, $\chi_2$, $x_0$, $H_0$, $\xi'$, $\beta$, $\beta'$, $\ve$, $q$, $a$, $r$, $k$, $n$, $y_i$ and $z_i$ exactly as in Section \ref{012}, considering $a$, $r$, $k$ and $n$ as functions of $q$. Then all steps go through, except the very last, where we used $\chi < 1/2$ to get a contradiction. Therefore all we need to do is the modify this last step to get a contradiction in a different way. This is where we need the sublinear variance inequality~\eqref{bksineq}. As before, throughout the proof $C$ denotes any constant that does not depend on  $q$.

For each real number $m\ge 1$, let $w_m$ be the nearest lattice point to $mx_0$. Note that $y_i = w_{ia}$. Let 
\[
f(m) := \var T(0, w_m).
\] 
Note that there is a constant $C_0$ such that $f(m) \le C_0m$ for all $m$. Again by \eqref{down1}, there is a $C_1>0$ such that for all $m$, 
\begin{equation}\label{fm}
f(m)\ge C_1m^{2\chi_1}.
\end{equation}
Now, $|(w_{(j+1)aq} - w_{jaq}) - w_{aq}| \le C$. Again, as a consequence of \eqref{referee1} we have that for any two random variables $X$ and $Y$, 
\begin{align}
\bigl|\var(X)-\var(Y)\bigr| &= \bigl|\sqrt{\var(X)} - \sqrt{\var(Y)}\bigr|\bigl(\sqrt{\var(X)} + \sqrt{\var(Y)}\bigr)\nonumber \\
&\le 2\|X-Y\|_2 \bigl(2\sqrt{\var(X)} + 2\|X-Y\|_2).\label{referee2}
\end{align}
By \eqref{referee2} and the subadditivity of first-passage times,
\begin{align*}
\var(T(w_{jaq}, w_{(j+1)aq})) &\ge f(aq) - C\sqrt{f(aq)} - C \\
&\ge f(n/r) - C\sqrt{n/r}. 
\end{align*}
Therefore by the Harris-FKG inequality, 
\begin{equation}\label{lowtt}
\var\biggl(\sum_{j=0}^{r-1} T(w_{jaq}, w_{(j+1)aq})\biggr) \ge r f(n/r) - C\sqrt{nr}. 
\end{equation}
Now, by \eqref{chib} and \eqref{chibep}, if $\chi_2$ is sufficiently close to $\chi$, then  both $1-\beta-\ve + \beta\chi_2$ and $1-\beta+\beta'\chi_2$ are strictly smaller than $1/2$. Therefore by~\eqref{tt}, \eqref{referee2} and the fact that $f(n)\le Cn$,
\begin{align*}
&\biggl|f(n) - \var\biggl(\sum_{j=0}^{r-1} T(w_{jaq}, w_{(j+1)aq})\biggr)\biggr|\\
&\le C\sqrt{n}(n^{1-\beta-\ve + \beta\chi_2}\log n +  n^{1-\beta + \beta' \chi_2}\log n).
\end{align*}
Combining this with \eqref{lowtt} gives
\begin{align*}
f(n) \ge r f(n/r) - C\sqrt{nr} - C\sqrt{n}(n^{1-\beta-\ve + \beta\chi_2}\log n +  n^{1-\beta + \beta' \chi_2}\log n).
\end{align*}
Again by \eqref{rbd} and \eqref{fm}, 
\begin{align*}
r f(n/r) \ge Cn^{(1-\beta-\ve) + (\beta+\ve)2\chi_1}. 
\end{align*} 
Combining \eqref{rbd} with the last two displays, it follows that we can choose $\chi_1$ and $\chi_2$ so close to $1/2$ that as $q\ra\infty$, 
\[
\liminf\frac{f(n)}{rf(n/r)} \ge 1.
\]
In particular, for any $\delta > 0$, there exists an integer $q(\delta)$ such that if $q\ge q(\delta)$, then 
\begin{equation}\label{fineq}
f(n) \ge (1-\delta) r f(n/r). 
\end{equation}
Fix $\delta = (1-\beta-\ve)/2$ and choose $q(\delta)$ satisfying the above criterion. Note that $q(\delta)$ can be chosen as large as we like. Let $m_0:= aq = n/r$ and $m_1 = n$. The above inequality implies that 
\[
\frac{f(m_1)}{m_1} \ge (1-\delta) \frac{f(m_0)}{m_0}.
\] 
Note that by \eqref{qbd},  if $q(\delta)$ is chosen sufficiently large to begin with, then 
\[
m_1^{\ve/(\beta+\ve)} > Cq^{1/(\beta+\ve)}> q(\delta).
\]
We now inductively define an increasing sequence $m_2, m_3,\ldots$ as follows. Suppose that $m_{i-1}$ has been defined such that 
\begin{equation}\label{qq}
m_{i-1}^{\ve/(\beta+\ve)} > q(\delta).
\end{equation}
Let 
\[
q_i := [m_{i-1}^{\ve/(\beta+\ve)}] + 1,
\]
where $[x]$ denotes the integer part of a real number $x$. By \eqref{qq}, $q_i \ge q(\delta)$. Let $a_i := m_{i-1}/q_i$. Then  if $q(\delta)$ is chosen large enough,
\[
a_i \ge \frac{2}{3}m_{i-1}^{\beta/(\beta+\ve)}\ge \frac{1}{2}q_i^{\beta/\ve},
\]
and 
\[
a_i \le m_{i-1}^{\beta/(\beta+\ve)}\le q_i^{\beta/\ve}.
\]
Let $r_i$ be an integer between $q_i^{(1-\beta-\ve)/\ve}$ and $2q_i^{(1-\beta-\ve)/\ve}$. Let $k_i = r_i q_i$ and $n_i = a_i k_i=a_i r_i q_i = r_i m_{i-1}$. If we carry out the argument of Section~\ref{012} with $q_i, r_i, k_i, a_i, n_i$ in place of $q,r,k,a,n$, then, since $q_i \ge q(\delta)$, as before we arrive at the inequality 
\[
f(n_i) \ge (1-\delta) r_i f(n_i/r_i) = (1-\delta) r_i f(m_{i-1}). 
\]
Define $m_i := n_i$. Then the above inequality shows that
\begin{equation}\label{mi}
\frac{f(m_i)}{m_i} \ge (1-\delta) \frac{f(m_{i-1})}{m_{i-1}}. 
\end{equation}
Note that since $r_i$ is a positive integer and $m_i = r_i m_{i-1}$, therefore $m_i \ge m_{i-1}$. In particular, \eqref{qq} is satisfied with $m_i$ in place of $m_{i-1}$. This allows us to carry on the inductive construction such that \eqref{mi} is satisfied for each~$i$.

Now, the above construction shows that if the initial $q$ was chosen large enough, then for each $i$,
\[
m_i=r_i m_{i-1} \ge q_i^{(1-\beta-\ve)/\ve} m_{i-1}\ge m_{i-1}^{1/(\beta+\ve)}. 
\]
Therefore, for all $i\ge 2$, 
\[
m_i \ge  m_1^{(\beta+\ve)^{-(i-1)}}. 
\]
So, by \eqref{bksineq}, there exists a constant $C_3$ such that
\[
\frac{f(m_i)}{m_i}\le \frac{C}{\log m_i} \le C_3(\beta+\ve)^{i-1}. 
\]
However, \eqref{mi} shows that there is $C_4 >0$ such that
\[
\frac{f(m_i)}{m_i}\ge C_4(1-\delta)^{i-1}. 
\]
Since $1-\delta > \beta+\ve$, we get a contradiction for sufficiently large $i$.

\section{Proof of $\chi \le 2\xi - 1$ when $\chi = 0$}\label{hard3}
As usual, we prove by contradiction. Assume that $\chi = 0$ and $2\xi -1 < \chi$. Then $\xi < 1/2$. Choose $\xi_1$, $\xi'$ and $\xi''$ such that $\xi < \xi_1 < \xi'' < \xi' < 1/2$. From this point on, however, the proof is quite different than the case $\chi > 0$. Recall that $t(P)$ is the sum of edge-weights of a path $P$ in the environment $t=(t_e)_{e\in E(\zz^d)}$. This notation is used several times in this section. First, we need a simple lemma about the norm $g$.
\begin{lmm}\label{glmm}
Assume that the edge-weight distribution is continuous, and let $L$ denote the infimum of its support. Then there exists $M > L$ such that for all $x\in \rr^d \backslash\{0\}$, $g(x) \ge M|x|_1$, where $|x|_1$ is the $\ell_1$ norm of $x$. 
\end{lmm}
\begin{proof}
Since $g$ is a norm on $\rr^d$, 
\[
M := \inf_{x\ne 0} \frac{g(x)}{|x|_1} > 0, 
\]
and the infimum is attained.  Choose $x\ne 0$ such that $g(x)=M|x|_1$. Define a new set of edge-weights $s_e$ as $s_e := t_e - L$. Then $s_e$ are non-negative and i.i.d. Let the function $g^s$ be defined for these new edge-weights the same way $g$ was defined for the old weights. Similarly, define $h^s$ and~$T^s$. Since any path $P$ from a point $y$ to a point $z$ must have at least $|z-y|_1$ many edges, therefore $s(P) \le t(P)-L|z-y|_1$. Thus, 
\[
T^s(y,z) \le T(y,z) - L|z-y|_1.
\]
In particular, $h^s(y) \le h(y) - L|y|_1$ for any $y$. Considering a sequence $y_n$ in $\zz^d$ such that $y_n/n \ra x$, we see that 
\begin{align*}
g^s(x) &= \lim_{n\ra\infty} \frac{h^s(y_n)}{n} \le \lim_{n\ra\infty} \frac{h(y_n)-L|y_n|_1}{n} \\
&= g(x)-L|x|_1 = (M-L)|x|_1.
\end{align*}
Since $t_e$ has a continuous distribution, $s_e$ has no mass at $0$. Therefore, by a well-known result (see~\cite{kesten86}), $g^s(x) > 0$. This shows that $M > L$. 
\end{proof}

Choose $\beta$, $\ve'$ and $\ve$ so small that $0< \ve' < \ve < \beta < (\xi''-\xi_1)/d$. Choose $x_0$ and $H_0$ as in Proposition \ref{curve}. Let $n$ be a positive integer, to be chosen arbitrarily large at the end of the proof. Again, as usual, $C$ denotes any positive constant that does not depend on our choice of $n$. 

Choose a point $z\in H_0$ such that $|z|\in [n^{\xi'}, 2n^{\xi'}]$. Let $v:= nx_0/2 + z$. Then by Proposition \ref{curve} and the fact that $\xi' < 1/2$, 
\begin{equation}\label{gvv1}
\begin{split}
|g(v) - g(nx_0/2)| &= (n/2) |g(x_0 + 2z/n)-g(x_0)|\\
&\le C|z|^2/n\le Cn^{2\xi' -1}\le C.  
\end{split}
\end{equation}
Similarly,
\begin{equation}\label{gvv2}
|g(nx_0-v) - g(nx_0/2)|\le Cn^{2\xi'-1}\le C. 
\end{equation}
Let $w$ be the closest lattice point to $v$ and let $y$ be the closest lattice point to~$nx_0$. Then $|w-v|$ and $|y-nx_0|$ are bounded by $\sqrt{d}$. Therefore inequalities~\eqref{gvv1} and~\eqref{gvv2} imply that 
\begin{equation}\label{gvv3}
|g(y)-(g(w)+g(y-w))| \le C. 
\end{equation}
Figure \ref{figvg} has an illustration of the relative locations of $y$ and $w$, together with some other objects that will be defined below.

By Theorem \ref{gapthm} and the assumption that $\chi = 0$, $|h(y)-g(y)|$, $|h(w)-g(w)|$ and $|h(y-w)-g(y-w)|$ are all bounded by $Cn^{\ve}$. Again by~\eqref{up1} of Theorem \ref{kpzthm} and the assumption that $\chi=0$,  the probabilities $\pp(|T(0,w) - h(w)| > n^{\ve})$, $\pp(|T(w,y)- h(y-w)|>n^\ve)$ and $\pp(|T(0,y) - h(y)| >  n^\ve)$ are all bounded by $e^{-Cn^{\ve-\ve'}}$. These observations, together with~\eqref{gvv3}, imply that there are constants $C_1$ and $C_2$, independent of our choice of $n$, such that 
\begin{equation}\label{ttt}
\pp(|T(0,y) - (T(0,w) + T(w, y))| > C_1n^\ve) \le e^{-C_2n^{\ve-\ve'}}. 
\end{equation}
Let $T_o(0,y)$ be the minimum passage time from $0$ to $y$ among all paths that do not deviate by more than $n^{\xi''}$ from the straight line segment joining $0$ and~$y$.  By assumption \eqref{up2} of Theorem~\ref{kpzthm}, 
\[
\pp(T_o(0,y) = T(0,y)) \ge 1- e^{-Cn^{\xi''-\xi_1}}. 
\]
Combining this with \eqref{ttt}, we see that if $E_1$ is the event 
\begin{equation}\label{e1def}
E_1 := \{|T_o(0,y) - (T(0,w) + T(w, y))| \le C_1n^\ve\},
\end{equation}
where $C_1$ is the constant from \eqref{ttt}, then there is a constant $C_3$ such that
\begin{equation}\label{e1}
\pp(E_1) \ge 1- e^{-C_3n^{\xi''-\xi_1}} - e^{-C_3n^{\ve-\ve'}}.
\end{equation}
Let $V$ be the set of all lattice points within $\ell_1$ distance $n^\beta$ from $w$. Let $\partial V$ denote the boundary of $V$ in $\zz^d$, that is, all points in $V$ that have at least one neighbor outside of $V$.
Let $w_1$ be the first point in $G(0,w)$ that belongs to $\partial V$, when the points are arranged in a sequence from $0$ to $w$. Let $w_2$ be the last point in $G(w,y)$ that belongs to $\partial V$, when the points are arranged in a sequence from $w$ to $y$. 
Let $G_1$ denote the portion of $G(0,w)$ connecting $w_1$ and $w$, and let  $G_2$ denote the portion of $G(w,y)$ connecting $w$ and $w_2$. Let $G_0$ be the portion of $G(0,w)$ from $0$ to $w_1$ and let $G_3$ be the portion of $G(w, y)$ from $w_2$ to $y$. Note that $G_0$ and $G_3$ lie entirely outside of $V$. Figure~\ref{figvg} provides a schematic diagram to illustrate the above definitions.

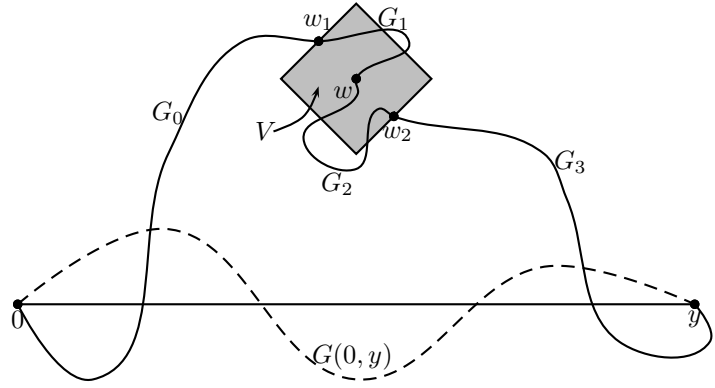
\begin{figure}[h]
\begin{pdfpic}
\begin{pspicture}(-1,-1)(10,5)
\psset{xunit=1cm,yunit=1cm}
\rput(0,-.2){\small $0$}
\rput(9,-.2){\small $y$}
\psline{*-*}(0,0)(9,0)
\pspolygon[fillstyle=solid, fillcolor=lightgray](4.5,4)(5.5,3)(4.5,2)(3.5,3)
\rput(4.33,2.82){\small $w$}
\psline{*-*}(4.5,3)(4.5,3)
\pscurve{*-*}(0,0)(1,-1)(2,2)(3,3.5)(4, 3.5)(5.2, 3.5)(4.5,3)(4.52,2.8)(3.8, 2.2)(4.5, 1.8)(4.8,2.6)(5,2.5)(7, 2)(7.3, 1.4)(8, -.5)(9.2,-.5)(9,0)
\psline{*-*}(4,3.5)(4,3.5)
\psline{*-*}(5,2.5)(5,2.5)
\rput(4,3.75){\small $w_1$}
\rput(5.05,2.27){\small $w_2$}
\rput(2,2.54){\small $G_0$}
\rput(5,3.8){\small $G_1$}
\rput(4.25, 1.6){\small $G_2$}
\rput(7.35, 1.9){\small $G_3$}
\pscurve{->}(3.4,2.3)(3.8, 2.5)(4.0,2.9)
\rput(3.3,2.3){\small $V$}
\pscurve[linestyle=dashed]{-}(0,0)(2,1)(4.5,-1)(7,.5)(9,0)
\rput(4.45,-.7){\small $G(0,y)$}
\end{pspicture}
\end{pdfpic}
\caption{Schematic diagram for $V, w, w_1,w_2$ and $G_0, G_1, G_2, G_3$.}
\label{figvg}
\end{figure}

Let $L$ and $M$ be as in Lemma \ref{glmm}. Choose $L', M'$ such that $L < L' < M' <M$. Take any $u\in \partial V$. By Lemma \ref{glmm}, $g(u-w) \ge M|u-w|_1$. Therefore by Theorem \ref{gapthm}, 
\[
h(u-w) \ge M|u-w|_1 - C|u-w|^\ve\ge M|u-w|_1 - Cn^{\beta \ve}. 
\]
Now, $|u-w|_1 \ge Cn^\beta$. Therefore by assumption \eqref{up1} of Theorem \ref{kpzthm} and the above inequality, 
\begin{align*}
&\pp(T(u,w) < M'|u-w|_1) \\
&\le \pp(|T(u,w) - h(u-w)| > (M-M')|u-w|_1 - Cn^{\beta \ve})\\
&\le \pp(|T(u,w) - h(u-w)| > Cn^{\beta})\le e^{-n^{\beta-\ve'}/C}.
\end{align*} 
Since there are at most $n^C$ points in $\partial V$, the above bound shows that 
\[
\pp(T(u, w) < M'|u-w|_1 \text{ for some } u\in \partial V) \le n^C e^{-n^{\beta-\ve'}/C}.
\]
In particular, if $E_2$ and $E_3$ are the events 
\begin{align*}
E_2 &:= \{t(G_1) \ge M'|w-w_1|_1\},\\
E_3 &:= \{t(G_2) \ge M'|w-w_2|_1\},
\end{align*}
then there is a constant $C_4$ such that 
\begin{align}
\pp(E_2\cap E_3) &\ge 1- n^{C_4}e^{-n^{\beta-\ve'}/C_4}. \label{te}
\end{align}
Let $E(V)$ denote the set of edges between members of $V$. Let $(t'_e)_{e\in E(V)}$ be a collection of i.i.d.\ random variables, independent of the original edge-weights, but having the same distribution. For $e\not\in E(V)$, let $t'_e := t_e$. Let $E_4$ be the event 
\[
E_4 := \{t_e' \le L' \text{ for each } e\in E(V)\}.
\]
If $E_4$ happens, then there is a path $P_1$ from $w_1$ to $w$ and  a path $P_2$ from $w$ to $w_2$ such that $t'(P_1) \le L'|w-w_1|_1$ and $t'(P_2)\le L'|w-w_2|_1$.  Let $P$ be the concatenation of the paths $G_0$, $P_1$, $P_2$ and $G_3$. Since $t'(G_0)=t(G_0)$ and $t'(G_3)=t(G_3)$, therefore under $E_4$,
\[
t'(P) \le t(G_0)+t(G_3) + L'|w-w_1|_1+L'|w-w_2|_1. 
\]
On the other hand, under $E_2\cap E_3$, 
\begin{align*}
T(0,w)+T(w,y) &= t(G_0)+t(G_1)+t(G_2)+t(G_3) \\
&\ge t(G_0)+t(G_3) + M'|w-w_1|_1 + M'|w-w_2|_1. 
\end{align*}
Consequently, if $E_1, E_2, E_3, E_4$ all happen simultaneously, then there is a (deterministic) positive constant $C_5$ such that
\[
T_o(0,y) \ge t'(P) + C_5 n^\beta - C_1 n^\ve,  
\]
where $C_1$ is the constant in the definition \eqref{e1def} of $E_1$. Since $\beta < \xi'' < \xi'$ and $x_0\not\in H_0$, the edges within distance $n^{\xi''}$ of the line segment joining $0$ and $y$ have the same weights in the environment $t'$ as in $t$. Since $\beta > \ve$, this observation and  the above display proves that $E_1\cap E_2\cap E_3 \cap E_4$ implies $D'(0,y) \ge n^{\xi''}$, where $D'(0,y)$ is the value of $D(0,y)$ in the new environment~$t'$. (To put it differently, if $E_1\cap E_2 \cap E_3 \cap E_4$ happens then there is a path $P$ that has less $t'$-weight than the least $t'$-weight path within distance $n^{\xi''}$ of the straight line connecting $0$ to $y$, and therefore $D'(0,y)$ must be greater than or equal to $n^{\xi''}$.)

Now note that the event $E_4$ is independent of $E_1$, $E_2$ and $E_3$. Moreover, since $L'>L$, there is a constant $C_6$ such that $\pp(E_4) \ge e^{-C_6 n^{\beta d}}$. Combining this with \eqref{e1}, \eqref{te} and the last observation from the previous paragraph, we get
\begin{align*}
\pp(D'(0,y)\ge n^{\xi''}) &\ge \pp(E_1\cap E_2\cap E_3 \cap E_4)\\
&= \pp(E_1\cap E_2\cap E_3)\pp(E_4)\\
&\ge (1-e^{-C_3 n^{\xi''-\xi_1}} - e^{-C_3 n^{\ve-\ve'}} - n^{C_4}e^{-n^{\beta-\ve'}/C_4}) e^{-C_6 n^{\beta d}}\\
&\ge e^{-C_7 n^{\beta d}}.  
\end{align*}
Now $D'(0,y)$ has the same distribution as $D(0,y)$. But by \eqref{up2} of Theorem~\ref{kpzthm}, $\pp(D(0,y) \ge n^{\xi''}) \le e^{-C_8n^{\xi''-\xi_1}}$, and $\beta d < \xi''-\xi_1$ by our choice of $\beta$. Together with the above display, this gives a contradiction, thereby proving that $\chi \le 2\xi -1$ when $\chi = 0$. 

\vskip.2in

\noindent {\bf Acknowledgments.} I would like to thank Tom LaGatta, Partha Dey, Elena Kosygina, Alain Sznitman, Rapha\"el Rossignol and Rongfeng Sun for useful discussions and references. I would also like to specially thank one of  the referees for a number of very useful suggestions and pointing out some errors in the first draft.

\end{document}